\documentclass{amsart}

\usepackage{enumerate}

\usepackage[utf8]{inputenc}
\usepackage{amsmath}
\usepackage{amsthm}
\usepackage{amsfonts}
\usepackage{amssymb}
\usepackage{color}

\theoremstyle{definition}
\newtheorem{definition}{Definition}

\newtheorem{theorem}{Theorem}
\newtheorem{corollary}{Corollary}
\newtheorem{proposition}{Proposition}
\newtheorem{lemma}{Lemma}

\theoremstyle{remark}
\newtheorem{remark}{Remark}

\renewcommand{\qedsymbol}{$\blacksquare$}

\begin{document}	
\title{Idealistic Flowers in the Reduction of Singularities}
\author{Felipe Cano}
\author{Beatriz Molina-Samper}
\address{Dpto. \'Algebra, An\'alisis Matem\'atico, Geometr\'ia y Topolog\'ia and Instituto de Matem\'aticas de la Universidad de Valladolid, Facultad de Ciencias, Universidad de Valladolid, Paseo de Bel\'en, 7, 47011 Valladolid, Spain}
\email{fcano@uva.es, beatriz.molina@uva.es}		
\date{\today}
\begin{abstract}
We present here a proof of the classical reduction of singularities based on the idea of ``idealistic flowers''. We follow the general ideas of Maximal Contact Theory, presented in a recent book of Aroca, Hironaka and Vicente, that recovers three old publications of Jorge Juan Institute. The concept of idealistic flowers deals with the globalization problems arising from the local nature of the maximal contact.
\end{abstract}
\maketitle
\setcounter{tocdepth}{2}
\tableofcontents

\section{Introduction}
One of the major difficulties in the reduction of singularities of complex analytic spaces, following Maximal Contact Theory, is the fact that not all the local permissible centers of blowing-up globalize. In the ori\-ginal proof \cite{Aro-H-V1,Aro-H-V2, Aro-H-V,Hir7}, the gluing of local centers is done with the help of the so called Hironaka's gardening, based on the groves and polygroves. In more recent proofs, the globalization is a consequence of functorial proper\-ties of the process.

Here we present another way to assure the globalization and we write a complete proof following it. The idea is to consider a class of objects, that we call Hironaka's flowers, or idealistic flowers, such that the usual induction on the dimension is well defined in this class.

Our starting objects are the idealistic spaces. They are composed of an ambient space and
a finite list of principal ideals with assigned multiplicity; the ambient space is a non-singular complex analytic space endowed of a normal crossings divisor. The singular locus is the set of points where the Hironaka's order is greater than or equal to one.
 The reduction of singularities of idealistic spaces implies the reduction of singularities of complex analytic spaces \cite{Hir4}. Thus, the objective is to show the existence of reduction of singularities for idealistic spaces.

 We say that two idealistic spaces over the same ambient space are equivalent when their transforms under permissible blowing-ups or the inverse transforms under open projections produce the same singular locus. More precisely, they have the same permissible test systems. In particular, a reduction of singularities of one of them gives a reduction of singularities for the other one.  Then, the problem of reduction of singularities makes sense for the classes under this equivalence.

 Once we fix an ambient space $(M,E)$, we can consider idealistic spaces over open subsets $(U,U\cap E)$ of the ambient space. We call them idealistic charts. Two idealistic charts are compatible when their restriction to the intersection of the domains gives two equivalent idealistic spaces. In this way, we can consider idealistic atlases for the ambient space $(M,E)$. The notion of singular locus and permissible centers for idealistic atlases makes sense and hence the problem of existence of reduction of singularities has also sense for idealistic atlases. Two idealistic atlases are compatible, or equivalent, when their union is also an idealistic atlas. The classes of compatibility of idealistic atlases are the idealistic exponents. Note that there is only one non-singular idealistic exponent over the ambient space $(M,E)$. The problem of existence of reduction of singularities for idealistic exponents is the same one as the problem of existence of reduction of singularities for idealistic atlases. A particular case of idealistic atlases is the one given by a single idealistic space; hence if the idealistic atlases admit reduction of singularities, then  we also have reduction of singularities for idealistic spaces.

 The Maximal Contact Theory intends to reduce the dimension of the ambient space, in the case of idealistic exponents, by projecting the problem of reduction of singularities over a special type of hypersurfaces, called maximal contact hypersurfaces. Such hypersurfaces will appear in the so-called adjusted and reduced case, but unfortunately they have not a global nature. The objects we naturally produce when we perform the projection are the idealistic flowers of dimension $e$, or idealistic $e$-flowers. In this way, we detect the class of objects where the classical induction on the dimension can be implemented, as well as other intermediate, but necessary, reductions of the problem. These objects are the idealistic $e$-flowers.

 Let us give the definition of idealistic $e$-flowers over an ambient space $(M,E)$. A transverse ambient $e$-subspace $(M,E,N)$ of $(M,E)$ is  given by a non-singular closed analytic subset $N\subset M$ having normal crossings with $E$ and being transverse to the components of the divisor $E$; it induces an ambient space $(N,E\vert_N)$. An immersed idealistic $e$-space over $(M,E)$ is the data of a transverse ambient $e$-subspace $(M,E,N)$ and an idealistic space $\mathcal N$ over $(N,E\vert_N)$. These objects have a well defined singular locus, the one of $\mathcal N$, and the permissible centers are subsets of $N$, and a fortiori subsets of $M$.  The transformations will be always done by restriction of transformations of the ambient space. Then, we can state  also the problem of existence of reduction of singularities for immersed idealistic $e$-spaces over $(M,E)$; note that when $e=\dim M$, we get the original situation of idealistic spaces. We extend the definition of equivalence of idealistic spaces to immersed idealistic $e$-spaces, by saying that two immersed idealistic $e$-spaces are equivalent when they have the same permissible test systems.

 Let us note that two equivalent immersed idealistic $e$-spaces over $(M,E)$ need not to be supported by the same transverse ambient $e$-subspace. Only the compatibility of the singular locus will be assured.

 An immersed idealistic $e$-chart  for $(M,E)$ is an immersed idealistic $e$-space over an open subset $(U,E\cap U)$ of $(M,E)$.  Two such charts are compatible when their restrictions to the intersection of the domains are equivalent. In this way, we can define immersed idealistic $e$-atlases. An idealistic $e$-flower over the ambient space $(M,E)$ is a class of compatibility of immersed idealistic $e$-atlases. Of course, the notion of singular locus is well defined, the idea of permissible center and the transforms under permissible blowing-ups. Thus, the problem of reduction of singularities for $e$-flowers makes sense and it extends naturally the problem of  reduction of singularities for idealistic exponents, hence for idealistic spaces.  Now, the basic induction statement is the following one
 \begin{equation*}
 \left.
 \begin{array}{c}
 \text{Reduction of singularities}\\\text{ for $(e-1)$-flowers}
 \end{array}
 \right.
 \Rightarrow
 \left.
 \begin{array}{c}
 \text{Reduction of singularities}\\\text{ for $e$-flowers}
 \end{array}
 \right.
 \end{equation*}

 Let us see how we organize the proof of this statement. First we reduce the problem to the adjusted and reduced case. ``Adjusted'' means that the Hironaka order at the singular points is exactly equal to one. The Hironaka order is well defined for $e$-flowers, but it is not stable for equivalent $e$-flowers of different dimensions. ``Reduced'' means that the singular locus does not contain hypersurfaces of the ambient subspaces in the corresponding immersed idealistic $e$-atlases. This reduction passes naturally through a combinatorial process of desingularization expressed in terms of polyhedra.

 Once we are dealing with an adjusted and reduced idealistic $e$-flower $\mathcal F$, we look for an idealistic $(e-1)$-flower $\mathcal H$, over the same ambient space having maximal contact with $\mathcal F$. That is, we ask  $\mathcal F$ and $\mathcal H$ to be equivalent in the sense that they have the same permissible test systems. For this, we develop a procedure of projection  of adjusted and reduced idealistic exponents over hypersurfaces, by using a tool, called projecting axis, inspired in some works of Panazzolo \cite{Pan}. Finally, we get locally and with an empty divisor the maximal contact hypersurface, just by the classical Tschirnhaus transformations. Once we put away from the singular locus the old components of the divisor $E$, we get the desired $(e-1)$-flower having maximal contact, whose reduction of singularities induces a reduction of singularities of $\mathcal F$.

The process we present here only looks for the existence of reduction of singularities, without asking additional properties of functoriality.  We have written this paper for introducing the idea of $e$-flower hoping that it would be helpful for the possible reduction of singularities for more complicate objects, as foliations, vector fields and other differential objects in characteristic zero.

A complete review of the problem of reduction of singularities is out of the scope of this paper. The old introduction to the problem of J. Lipman \cite{Lip} is a good reference and in the paper \cite{Spi1} the reader may find an extensive bibliography. Anyway, let us cite the papers \cite{Cos3,Cos-P2} for the positive characteristic case, the book \cite{Cut} on monomialization of morphisms,  the papers \cite{Can3, Can-R-S, McQ-P, Pan} in the case of differential objects, the papers \cite{Abr-J, Abra-T-W, Mar-M} for the use of weighted blowing-ups. The approaches through Zariski Local Uniformization \cite{Zar} and Hironaka's Voûte Étoilée \cite{ Cut2, Hir6} are also very important and the open questions in positive characteristic and the differential cases could maybe approached through these ideas, see for instance 
\cite{Can-F, Nov-S, Pil}.

We are focusing the classical reduction of singularities of complex spaces in the Maximal Contact approach.  There are many references to this approach, we find specially interesting the introduction of B. Teissier to the book \cite{Aro-H-V},  the reader may also look at 
\cite{Enc-H,Gir2,Hau,Kol, Mus,Spi1,Wlo2}
 and others.
 Let us note that in \cite{Bie-M} and \cite{Vil} appeared for the first time a procedure with strong additional properties of constructibility and functoriality. Recently, the globalization of the maximal contact has been based in an extensive use of Giraud's approximation to maximal contact \cite{Gir1}, as well as in  additional properties of uniqueness for the obtained desingularization process. We follow another way, by using the idealistic $e$-flowers, where the globalization is given thanks to the omnipresence of a global ambient space.

{\em Acknowledgements:} Both authors, we are very grateful to J.M. Aroca that has transmitted to us the ideas of reduction of singularities along a whole professional life. The first author is also very grateful to Jean Giraud, Mark Spivakovsky, Bernard  Teissier and Vincent Cossart for many fruitful conversations on the subject.

\part{Objects and Statements}
\label{Objects and Statements}
We present here the main concepts and statements in this work.
\section{Ambient Spaces}

Our {\em ambient spaces} are pairs $((M,K),(E,E\cap K))$, where the space $(M,K)$ is the germ of a non-singular complex analytic space $M$ over a compact subset $K\subset M$ and $(E,E\cap K)$ is the germ over $E\cap K$ of a normal crossings divisor $E\subset M$. If there is no possible confusion, we simply denote $(M,E)$ the ambient space.

An {\em open subset} of $(M,K)$ is an open immersion of germs
$$
(U,L)\subset (M,K),
$$
where $L\subset K$ is a compact. We simply denote $U\subset M$. Let us remark that we always have that $(U,L)=(M,L)$, viewed as germified spaces.

%
%
%An {\em open ambient subspace of $(M,E)$} is a pair $(U,E\cap U)$, where $U\subset M$ is an open set, germified over $L\subset K$, and $U\cap E$ stands for the germ $(E,E\cap L)$.

An {\em open covering} of $(M,E)$ is a family $\{(M_\alpha,E_\alpha)\}_{\alpha \in \Lambda}$ of open ambient spaces, where the $M_\alpha$ are open subsets of $M$ such that $E_\alpha=E\cap M_\alpha$ and the compacts $K_\alpha$ of germification for $M_\alpha$ satisfy that $K=\cup_{\alpha\in \Lambda} K_\alpha$.

A {\em closed analytic subset of} $(M,K)$ is a closed immersion of germs $(N,K\cap N) \subset (M,K)$, where $N\subset M$ is a closed analytic subset.

An {\em $e$-dimensional closed ambient subspace $(M,E,N)$} of an ambient space $(M, E)$ is given by a purely $e$-dimensional non-singular closed analytic subset $N$ of $M$ having normal crossings with $E$. In this way, we obtain an $e$-dimensional ambient space $(N,E\vert_N)$, where $E|_N$ denotes the
union of the intersections with $N$ of the irreducible components of $E$ not containing $N$ (locally). When $N$ is not locally contained in any of the irreducible components of $E$, we say that $(M,E,N)$ is a {\em transverse closed ambient subspace}. In this case we have that $E|_N=E \cap N$.

 A  {\em hypersurface $(M,E,H)$ of the ambient space $(M,E)$} is just a closed ambient subspace of $(M,E)$ such that
$\dim H=\dim M-1$. In this case, we have $E\vert_H=E^*\cap H$, where $E^*$ is the union of the irreducible components of $E$ not contained in $H$; the hypersurface is transverse if and only if $E^*=E$.

\begin{remark}
Let $(N,K\cap N)$ be a closed analytic subset of $(M,K)$. Sometimes we refer to the complement  $M\setminus N$. The complement does not have the nature of a germ of analytic space over a compact subset; its interpretation must be done in terms of appropriate representatives. Sometimes we consider the restriction to such complements of sheaves and other objects; the interpretation must be done in terms of points ``close enough'' to the compacts of germification.
We hope that these notations will not produce confusion and, on the other hand, they will contribute to simplify the exposition.
\end{remark}
\strut

\section{Idealistic Spaces}
A {\em marked principal ideal} $\mathcal I$ over an ambient space $(M,E)$ is a pair $\mathcal I=(I,d)$, where $I\subset \mathcal O_M$ is a sheaf of principal ideals and $d$ is a positive integer number that we call the {\em assigned multiplicity of $\mathcal I$}. A point $P\in M$ is a {\em singular point for} $\mathcal I$ if $\nu_PI\geq d$, where $\nu_PI$ stands for the multiplicity of $I$ at $P$. The {\em singular locus} $\operatorname{Sing}\mathcal I$ is the closed analytic subset of $M$ given by the singular points. Since $I_P\ne 0$ for any point $P$, the dimension of the singular locus is strictly lower than the dimension of $M$.

An  {\em idealistic space $\mathcal M$ over $(M,E)$} is a triple $\mathcal M=(M,E,\mathcal L)$, where $\mathcal L$ is a finite list of marked principal ideals
$$
\mathcal L=\{\mathcal I_j=(I_j,d_j)\}_{j=1}^k.
$$
The {\em singular locus $\operatorname{Sing}\mathcal L$} is defined by $\textstyle
\operatorname{Sing}\mathcal L=\bigcap_{j=1}^k\operatorname{Sing}\mathcal I_j$. We also say that  $\operatorname{Sing}\mathcal L$ is the  {\em singular locus of  $\mathcal M$} and we write $\operatorname{Sing}\mathcal M=\operatorname{Sing}\mathcal L$.
\begin{remark} For practical reasons, we can admit zero ideals in the list $\mathcal L$, but we always ask that at least one of the marked ideals is non-null. More precisely, a list with some zero ideals represents by definition the same idealistic space as the list obtained by skipping all the zero ideals.
\end{remark}
\subsection{Transformations of Idealistic Spaces}
We consider two types of transformations $\sigma:(M',E')\rightarrow (M,E)$
of the ambient space $(M,E)$: open projections and blowing-ups with non-singular centers having normal crossings with $E$  (the transformation by isomorphisms is evident and we will not insist on that). When we apply a blowing-up to an idealistic space $\mathcal M$, we ask the center $Y$ to be {\em permissible} in the sense that it is contained in the singular locus of $\mathcal M$.

Let $\mathcal M=(M,E,\mathcal L=\{(I_j,d_j)\}_{j=1}^k)$ be an idealistic space and let us denote
$
\sigma:(M',E')\rightarrow (M,E),
$
one of that morphisms. {\em The transform
$$
\mathcal M'=(M',E',\mathcal L'=\{(I'_j,d_j)\}_{j=1}^k)
$$
of $\mathcal M$ by $\sigma$} is given as follows, depending on the nature of $\sigma$:

If $\sigma$ is an open inclusion the {\em restriction  $\mathcal M'=\mathcal M\vert_{M'}$ of $\mathcal M$ to $M'$} is defined in an evident way. Assume that $\sigma$ is a projection over the first factor
$$
\sigma: (M',E')=(M\times (\mathbb C^m,0), E\times (\mathbb C^m,0))\rightarrow (M,E).
$$
We take $I'_j=\sigma^{-1}I_j$, for $j=1,2,\ldots,k$. The singular locus satisfies that
$$
\operatorname{Sing}{\mathcal M'}=
(\operatorname{Sing}{\mathcal M})\times (\mathbb C^m,0).
$$
Note that the projection contains, as a datum, the functions
$$
\omega_i: M'\rightarrow (\mathbb C,0),\quad i=1,2,\ldots,m,
$$ obtained from the natural projections $({\mathbb C^m,0})\rightarrow(\mathbb C,0)$.
We say that $\sigma$ is an {\em open projection} if it is composition of an open inclusion and a projection over the first factor. We define the transform $\mathcal M'$ by making first the transform by the open inclusion and secondly the transform by the projection.

Assume that $\sigma$ is the blowing-up of $(M,E)$ with a permissible center $Y\subset M$. That is, the morphism $\sigma$ is given by the blowing-up $M'\rightarrow M$ with center $Y$, where we take $E'=\sigma^{-1}(E\cup Y)$. Each $I'_j$ is
{\em the controlled transform of $I_j$ with assigned multiplicity $d_j$}, for $j=1,2,\ldots,k$. This means that
$$
I'_j= \mathcal J_{D'}^{-d_j}\pi^{-1}(I_j),\quad D'=\pi^{-1}(Y),
$$
where $\mathcal J_{D'}\subset \mathcal O_{M'}$ is the ideal sheaf of the exceptional divisor $D'$.

 Note that in all cases $\mathcal L'$ has the same assigned multiplicities as $\mathcal L$.
\subsection{Reduction of Singularities of Idealistic Spaces}
In the context of idealistic spaces, the classical Hironaka's reduction of singularities may be stated as follows:
\begin{theorem} \label{teo:Ingl redsing}
Let $\mathcal M$ be an idealistic space over the ambient $(M,E)$. There is a morphism $\sigma:(M',E')\rightarrow (M,E)$ that is composition of a finite sequence of blowing-ups with permissible centers, such that
$$
\operatorname{Sing}\mathcal M'=\emptyset,
$$
where $\mathcal M'$ is the transform of $\mathcal M$ by $\sigma$.
\end{theorem}
In other words, there is a {\em reduction of singularities} for any idealistic space $\mathcal M$. This result will be a consequence of the existence of reduction of singularities for more general objects that we call  {\em idealistic flowers}.

\section{Test Systems and Equivalence of Idealistic Spaces}
Consider an ambient space $(M,E)$. A {\em test system $\mathcal S$ over $(M,E)$ of length $k\geq 1$} is a family  $\mathcal S=\{\left(Y_{j-1}, \sigma_j\right)\}_{j=1}^k$ with
$$
\sigma_j:(M_j,E_j)\rightarrow (M_{j-1},E_{j-1}),
$$
where $(M_0,E_0)=(M,E)$, each  $Y_{j-1}$ is the empty set or a non-singular closed analytic subset $Y_{j-1}\subset M_{j-1}$ having normal crossings with $E_{j-1}$.
If  $Y_{j-1}=\emptyset$, then $\sigma_j$ is an open projection. If $Y_{j-1}\ne \emptyset$, then $\sigma_j$ is the blowing-up with center $Y_{j-1}$. A {\em test system of length $0$} is just the identity $(M,E)\rightarrow (M,E)$, understood as an open projection.

For each  $0\leq \ell\leq k$, we define the {\em truncation $\mathcal S^{\ell}$} to be obtained from $\mathcal S$, by taking just the indices $j\leq\ell$.

\subsection{Restriction of a Test System to an Open Subset}
Let $\mathcal S$ be a test system over $(M,E)$ of length $k$. Consider a non-empty open set $U\subset M$. Let us define the {\em restriction $\mathcal S_U$ of $\mathcal S$ a $U$}. It is a test system over $(U,U\cap E)$ of length  $k'\leq k$, that we write
$$
\mathcal S_U=\{( Y'_{j-1}, \sigma'_j)\}_{j=1}^{k'},\quad
\sigma'_j:(M'_{j},E'_j)\rightarrow (M'_{j-1},E'_{j-1}),
$$
such that the following properties hold:
\begin{enumerate}
\item $(M'_0,E'_0)=(U,U\cap E)$.
\item If $j<k'$, then  $M'_{j+1}=\sigma_{j}^{-1}(M'_{j})\ne\emptyset$.
\item If $k'<k$, then $\sigma_{k'+1}^{-1}(M'_{k'})=\emptyset$.
\item We have that $Y'_{j-1}=M'_{j-1}\cap Y_{j-1}$, for any  $1\leq j\leq k'$.
\begin{itemize}
	\item [-] If $Y'_{j-1}\ne\emptyset$, then $\sigma'_{j}$ is the blowing-up of $(M'_{j-1},E'_{j-1})$ with center $Y'_{j-1}$.
	\item [-] If $Y'_{j-1}=\emptyset$, then $\sigma'_{j}$ is the restriction $(M'_{j},E'_{j})\rightarrow (M'_{j-1},E'_{j-1})$ of $\sigma_j$, understood as an open projection.
\end{itemize}
\end{enumerate}
Let us remark that  $M'_j\subset M_j$ is an open set and $\sigma'_j$ is a restriction of  $\sigma_j$.

\begin{remark}
The length  $k'$ of $\mathcal S_U$ is smaller or equal than $k$.
For practical reasons, we introduce the truncation $\mathcal S^\ell$ with respect to $\ell\geq k$ to be given by $\mathcal S^\ell=\mathcal S$. In this situation, we have that
\begin{equation} \label{eq:restringirtruncar}
	(\mathcal S_U)^{k-1}=(\mathcal S^{k-1})_U=\mathcal S^{k-1}_U.
\end{equation}
\end{remark}

\subsection{Permissible Test Systems}
\label{Permissible Test Systems}

Let us consider an idealistic space  $\mathcal M$ and a test system $\mathcal S$ of length $k\geq 0$ over $(M,E)$. We define, by induction on the length $k$, the concept of test systems that are {\em permissible  for $\mathcal M$}, or {\em $\mathcal M$-permissible test systems}, and the concept of {\em transform of $\mathcal M$ by $\mathcal M$-permissible test systems}.

If $k=0$, the test system is permissible and the transform of $\mathcal M$ is $\mathcal M$ itself. Assume that  $k\geq 1$. We say that $\mathcal S$ is permissible for $\mathcal M$ if $\mathcal S^{k-1}$ is permissible for $\mathcal M$ and the following hold: either $Y_{k-1}=\emptyset$, or $Y_{k-1}\ne\emptyset$ and it is permissible for the transform $\mathcal M_{k-1}$ of $\mathcal M$ by $\mathcal S^{k-1}$. In both cases, we define the transform $\mathcal M_k$ to be the transform of $\mathcal M_{k-1}$ by $\sigma_k$, where we consider $\sigma_k$ as an open projection if $Y_{k-1}=\emptyset$ or as a blowing-up, otherwise.

The concept of permissible center is of a local nature. In the same way, the concept of permissible test system is also of a local nature, as it is stated in the next proposition, that is a direct consequence of the local character of permissible blowing-up centers.

\begin{proposition} \label{prop:ingl proplocalsisprueba}
Consider an idealistic space $\mathcal M$ and a test system  $\mathcal S$ over $(M,E)$. Let  $\{(M_\alpha,E_\alpha)\}_{\alpha\in \Lambda}$ be an open covering of $(M,E)$. The test system  $\mathcal S$ is permissible for $\mathcal M$ if and only if the restriction  $\mathcal S_{M_\alpha}$ is permissible for   $\mathcal M\vert_{M_\alpha}$, for each $\alpha\in \Lambda$.
\end{proposition}

\subsection{Equivalent Idealistic Spaces}
Two idealistic spaces $\mathcal M_1$ and $\mathcal M_2$ over the same ambient space $(M,E)$ are called to be {\em equivalent} if they have exactly the same permissible test systems.

A direct consequence of this definition is that two equivalent idealistic spaces have the same singular locus and that their transforms by a permissible test system are equivalent idealistic spaces. In particular, if $\mathcal M_1$ and $\mathcal M_2$ are equivalent, a reduction of singularities of $\mathcal M_1$ is also a reduction of singularities of $\mathcal M_2$, and conversely.

Next statement concerns to the local character of the equivalence between idealistic spaces:

\begin{proposition}
\label{prop:ingl caracterlocalequivalencia}
Consider two idealistic spaces 	$\mathcal M_1$ and $\mathcal M_2$ over $(M,E)$ and an open covering $\{(M_\alpha,E_\alpha)\}_{\alpha \in \Lambda}$ of $(M,E)$. We have that $\mathcal M_1$ and $\mathcal M_2$ are equivalent if and only if the restrictions $\mathcal M_1|_{M_\alpha}$ and $\mathcal M_2|_{M_\alpha}$ are equivalent, for all $\alpha\in \Lambda$.
\end{proposition}
\begin{proof}
Assume that $M_1|_{M_\alpha}$ and $\mathcal M_2|_{M_\alpha}$ are equivalent, for all $\alpha\in \Lambda$. Take an $\mathcal M_1$-permissible test system $\mathcal S$. By Proposition \ref{prop:ingl proplocalsisprueba} we have that $\mathcal S_{M_\alpha}$ is permissible for $\mathcal M_1|_{M_\alpha}$ and hence for $\mathcal M_2|_{M_\alpha}$, thus $\mathcal S$ is permissible for $\mathcal M_2$.

Suppose that $\mathcal M_1$ and $\mathcal M_2$ are equivalent. Take $\alpha \in \Lambda$ and a test system $\mathcal S$ that is permissible  for $\mathcal M_1|_{M_\alpha}$. Let $\mathcal S^*$ be the test system over $(M,E)$ obtained by adding to $\mathcal S$ the inclusion $M_\alpha \subset M$ as the first element. Since the transform of $\mathcal M_1$ by the inclusion $M_\alpha \subset M$ is exactly $\mathcal M_1|_{M_\alpha}$, we get that $\mathcal S^*$ is permissible for $\mathcal M_1$ and hence for $\mathcal M_2$. The fact that $\mathcal S^*$ is permissible for $\mathcal M_2$ implies that $\mathcal S$ is permissible for $\mathcal M_2|_{M_\alpha}$.
\end{proof}

\subsection{Examples of Equivalent Idealistic Spaces}
\label{ingl Ejemplos de espacios idealisticos equivalentes}
Here we consider some useful examples of equivalent idealistic spaces.

\paragraph{ \em Normalization of an idealistic space.}
We say that an idealistic space is {\em normalized} when all the assigned multiplicities are the same ones. Let $\mathcal M=(M, E, \mathcal L)$ be an idealistic space, where $\mathcal L=\{(I_j,d_j)\}_{j=1}^k$ and take a common multiple $d$ of the  $d_j$. The idealistic space
$$
\mathcal M'=(M,E,\{(I'_j,d)\}_{j=1}^k), \quad I'_j=(I_j)^{d/d_j}
$$
is normalized and it is equivalent to $\mathcal M$.

\paragraph{\em Redundant marked ideals.}
Let $\mathcal M=(M, E, \mathcal L)$ and $\mathcal M'=(M,E,\mathcal L')$ be two idealistic spaces. Assume that $\mathcal L=\{(I_j,d_j)\}_{j=1}^k$ and $\mathcal L'=\cup_{j=1}^k\mathcal L_{j}$, where
$$
\mathcal L_j=\{(I_j,d_j)\}\cup \{(I_{js},d_j)\}_{s=1}^{k_j},
$$
with $I_{js}\subset I_j$, for all $s=1,2,\ldots,k_j$. Then $\mathcal M$ and $\mathcal M'$ are equivalent idealistic spaces.

\subsection{Infinite Test Systems}
An {\em infinite test system  $\mathcal S^\infty$ over $(M,E)$} is an infinite sequence
$$
\mathcal S^\infty=\{(Y_{j-1}, \sigma_j)\}_{j=1}^\infty
$$
such that the truncation $\mathcal S^k=\{(Y_{j-1}, \sigma_j)\}_{j=1}^k$ of $\mathcal S^\infty$ is a test system over $(M,E)$ of length $k$, for any $k\geq 0$. We say that $\mathcal S^\infty$ is {\em permissible for an idealistic space $\mathcal M$} if $\mathcal S^k$ is permissible for $\mathcal M$, for any $k\geq 0$.

\section{Idealistic Atlases and Idealistic Exponents} \label{ingl Atlasidealisticos}
In this section we introduce the concept of idealistic exponent. It is defined in terms of equivalence classes of idealistic atlases, in a parallel way to the classical language of Differential Geometry. Anyway, the idealistic atlases have their own interest and sometimes we will work with specific types of idealistic atlases belonging to a given idealistic exponent.
\subsection{Idealistic Atlases}
Let us consider an ambient space $(M,E)$. An {\em idealistic atlas $\mathcal A$ over $(M,E)$} is a finite family $\mathcal A=\{\mathcal M_\alpha\}_{\alpha\in \Lambda}$ such that:
\begin{enumerate}
	\item The $\mathcal M_\alpha$ are idealistic spaces over open sets $(M_\alpha,E_\alpha)\subset (M,E)$, where $\{(M_\alpha,E_\alpha)\}_{\alpha \in \Lambda}$ is an open covering of $(M,E)$.
	\item ({\em Compatibility property}) For any pair of indices $\alpha,\beta\in \Lambda$, the idealistic spaces
	$$
	\mathcal M_{\alpha\beta}=\mathcal M_\alpha\vert_{M_{\alpha\beta}}, \quad
	\mathcal M_{\beta\alpha}=\mathcal M_\beta\vert_{M_{\beta\alpha}}
	$$
	are equivalent, where $M_{\alpha\beta}=M_{\beta\alpha}=M_\alpha\cap M_\beta$.
\end{enumerate}
Each idealistic space $\mathcal M_\alpha$ will be called an {\em idealistic chart of $\mathcal A$}. More generally,  an {\em idealistic chart over $(M,E)$} is any idealistic space of the form $(U,U\cap E, \mathcal L)$, where $U$ is an open subset of $M$.

Assume that $\mathcal A$ is an idealistic atlas over $(M,E)$ and denote by $S_\alpha$ the singular locus of $\mathcal M_\alpha$. Denote also $S_{\alpha\beta}=S_\alpha\cap M_{\alpha\beta}$, we know that
$S_{\alpha\beta}=\operatorname{Sing}(\mathcal M_{\alpha\beta})$. Since $\mathcal M_{\alpha\beta}$ is equivalent to $\mathcal M_{\beta\alpha}$, we have that
$S_{\beta\alpha}=S_{\alpha\beta}$. This allows us to gluing together the singular loci $S_\alpha$ in a closed analytic subset $S\subset M$ such that $S\cap M_\alpha=S_\alpha$, for all $\alpha\in \Lambda$. We say that $S$ is the {\em singular locus of $\mathcal A$} and we denote it as $S=\operatorname{Sing}\mathcal A$.

The transformations of $\mathcal A$ by restriction to an open set of $M$ and by a projection over the first factor are directly defined from the case of idealistic spaces.

A {\em permissible center for $\mathcal A$} is a non-singular closed analytic subset of $\operatorname{Sing}\mathcal A$ having normal crossings with $E$. The local character of the permissible centers, expressed in next Proposition \ref{prop:caracterlocalcentrosatlas}, is a consequence of the local nature of the equivalence between idealistic spaces given in Proposition \ref{prop:ingl caracterlocalequivalencia}.
\begin{proposition}
\label{prop:caracterlocalcentrosatlas}
A closed analytic subset $Y$ of $M$ is a permissible center for $\mathcal A$ if and only if the following equivalent properties hold:
\begin{enumerate}
	\item For any point $P\in M$, there is an open subset $U\subset M$, with $P\in U$, such that $Y\cap U$ is a permissible center for $\mathcal A\vert_U$.
	\item For any open subset $U\subset M$, the intersection $Y\cap U$ is a permissible center for $\mathcal A\vert_U$.
	\item Given $\alpha \in \Lambda$, the intersection $Y\cap M_\alpha$ is permissible for $\mathcal M_\alpha$.
\end{enumerate}
\end{proposition}

Consider a blowing-up
$\sigma:(M',E')\rightarrow (M,E)$
with permissible center $Y\subset M$. For each $\alpha\in \Lambda$, the restriction of $\sigma$ is the blowing-up
$$
\sigma_\alpha:(\sigma^{-1}(M_\alpha), E'_\alpha=E'\cap \sigma^{-1}(M_\alpha) )\rightarrow
(M_\alpha,E\cap M_\alpha)
$$
of $(M_\alpha,E\cap M_\alpha)$ with center $Y_\alpha=Y\cap M_\alpha$ (the identity when $Y_\alpha=\emptyset$). Note that $Y_\alpha$ is a permissible center for $\mathcal M_\alpha$. Let $\mathcal M'_{\alpha}$ be the transform of $\mathcal M_\alpha$ by $\sigma_\alpha$. Given two indices $\alpha,\beta\in \Lambda$, we have an induced blowing-up of $M_{\alpha\beta}$ with center $Y_{\alpha\beta}$, this implies that $\mathcal M'_{\alpha\beta}$ and $\mathcal M'_{\beta\alpha}$ are equivalent. In this way, we define the {\em  transform $\mathcal A'$ of $\mathcal A$ by $\sigma$} to be the idealistic atlas over $(M',E')$ given by the family of the $\mathcal M_\alpha'$.

\subsection{Idealistic Exponents}
Consider an idealistic atlas $\mathcal A$ over the ambient space $(M,E)$. In the same way as for the case of idealistic spaces in Subsection \ref{Permissible Test Systems}, we can define the concept of test system that is {\em permissible for $\mathcal A$}, or {\em $\mathcal A$-permissible}, and the {\em transform of $\mathcal A$} by an $\mathcal A$-permissible test system.

As in Proposition \ref{prop:ingl proplocalsisprueba}, the property of being permissible a test system for an idealistic atlas has local nature as follows:

\begin{proposition}
\label{prop:ingl sispruebapermitidos}
Consider an idealistic atlas $\mathcal A$ over the ambient space $(M,E)$. A test system $\mathcal S$ over $(M,E)$ is permissible for $\mathcal A$ if and only if the following equivalent statements hold:
\begin{enumerate}
	\item For any point $P\in M$, there is an open subset $U\subset M$, with $P\in U$, such that $\mathcal S_U$ is a permissible test system for $\mathcal A\vert_U$.
	\item For any open subset $U\subset M$, the restriction $\mathcal S_U$ of $\mathcal S$ to $U$ is permissible for $\mathcal A\vert_U$.
	\item For any idealistic chart $\mathcal M_\alpha=(M_\alpha,E_\alpha,\mathcal L_\alpha)$ of $\mathcal A$, the restricted test system $\mathcal S_{M_\alpha}$ is permissible for $\mathcal M_\alpha$.
	\end{enumerate}
\end{proposition}
\begin{proof}
Follows from Proposition \ref{prop:caracterlocalcentrosatlas}.
\end{proof}

An idealistic chart $\mathcal C$ over $(M,E)$ is {\em compatible} with the idealistic atlas $\mathcal A$ if $\mathcal A\cup \{\mathcal C\}$ is again an idealistic atlas.

\begin{proposition} \label{prop:ingl equivalencia de atlas}
Given two idealistic atlases $\mathcal A_1$ and $\mathcal A_2$ over $(M,E)$, the following properties are equivalent:
\begin{enumerate}
	\item Any idealistic chart of $\mathcal A_2$ is compatible with $\mathcal A_1$.
	\item Any idealistic chart of $\mathcal A_1$ is compatible with $\mathcal A_2$.
	\item The union $\mathcal A_1\cup \mathcal A_2$ is an idealistic atlas.
	\item $\mathcal A_1$ and $\mathcal A_2$ have the same permissible test systems.
	\end{enumerate}
\end{proposition}
\begin{proof}
The statement comes from Proposition \ref{prop:ingl sispruebapermitidos}.
\end{proof}

\begin{definition}
Two idealistic atlases $\mathcal A_1$ and $\mathcal A_2$ over $(M,E)$ are called to be {\em equivalent} if $\mathcal A_1$ and $\mathcal A_2$ have the same permissible test systems (and hence the equivalent properties in Proposition \ref{prop:ingl equivalencia de atlas} hold).
\end{definition}

\begin{definition}
The equivalence classes of idealistic atlases over $(M,E)$ are called  {\em idealistic exponents over $(M,E)$}.
\end{definition}

The properties and concepts invariant by the equivalence relation of the idealistic atlases define properties and concepts concerning idealistic exponents $\mathcal E$ over $(M,E)$. Thus, we have well defined:
\begin{itemize}
	\item Singular locus $\operatorname{Sing}\mathcal E$.
	\item Concept of permissible centers.
	\item Transformations by open projections.
	\item Transformations by blowing-ups of permissible centers.
	\item Concept of $\mathcal E$-permissible test systems.
	\item Transformations by $\mathcal E$-permissible test systems.
	\item Existence of reduction of singularities.
\end{itemize}

\begin{remark}
In next parts we will introduce other concepts concerning idealistic exponents. The most relevant among them will be the ``order'' and the property of ``being monomial''.
\end{remark}
Let us note that there is only one non-singular idealistic exponent over a given ambient space.	So, the problem of the existence of reduction of singularities for idealistic exponents consists on finding a finite sequence of permissible blowing-ups to obtain the non-singular idealistic exponent from a given one. Note also, that the existence of reduction of singularities for idealistic exponents implies Theorem \ref{teo:Ingl redsing}.

\section{Immersed Idealistic Spaces}
\label{Immersed Idealistic Spaces}
Let $(M,E)$  be an ambient space. An {\em immersed idealistic $e$-space $\mathcal V$ over $(M,E)$} is a datum
$$
\mathcal V=(M,E,N,\mathcal L),
$$
where $(M,E,N)$ is a transverse $e$-dimensional closed ambient subspace of $(M,E)$ and $\mathcal N=(N,E\vert_N,\mathcal L)$ is an idealistic space over $(N,E\vert_N)$. We also say that $\mathcal V$ is an {\em immersed idealistic space of dimension $e$}.

The {\em singular locus $\operatorname{Sing}\mathcal V$} is, by definition, the singular locus of $\mathcal N$. It is a closed analytic subset of $N$ and hence it is also a closed analytic subset of $M$. The permissible centers of $\mathcal V$ are, by definition, the permissible centers of $\mathcal N$; they have normal crossings with $E$ in $M$.

We need to include the possibility that  $N=\emptyset$. In this case, we postulate the existence of a unique immersed idealistic $e$-space
$$
\mathcal V_\emptyset=(M,E,\emptyset,\mathcal L_\emptyset),
$$
where $\mathcal L_\emptyset$ is empty. The singular locus of $\mathcal V_\emptyset$ is also the empty set.

The transformations of an immersed idealistic $e$-space $\mathcal V$ are defined from morphisms of the ambient space $(M,E)$. Let us precise them.

The {\em restriction $\mathcal V\vert_U$}  to an open subset $U\subset M$ is given by
$$
\mathcal V\vert_U=(U,E\cap U,N\cap U,\mathcal L\vert_{N\cap U}).
$$
A projection over the first factor $\sigma: M\times(\mathbb C^m,0)\rightarrow M$ defines a projection over the first factor $\bar{\sigma}:N\times(\mathbb C^m,0)\rightarrow N$. In this way, we define the {\em transform $\mathcal V'=(M',E',N',\mathcal L')$ of $\mathcal V$ by $\sigma$}, by putting
$$
M'=M\times(\mathbb C^m,0),\quad  E'=E\times(\mathbb C^m,0), \quad N'=N\times(\mathbb C^m,0)
$$
and we take $\mathcal L'$ to be the transform of $\mathcal L$ by $\bar{\sigma}$.

Let $Y$ be a permissible center for  $\mathcal V$. Recall that we have $Y\subset N\subset M$.
Consider the blowing-up $\pi:(M',E')\rightarrow (M,E)$ with center $Y$ and denote by $N'$ the strict transform of $N$. Note that $(M',E',N')$ is a transverse closed ambient subspace of $(M',E')$.
The restriction of $\pi$ induces the blowing-up
$$
\bar\pi:(N',E'\vert_{N'})\rightarrow (N,E\vert_{N})
$$
of $(N,E\vert_N)$ with center $Y$. We define the {\em transform $\mathcal V'$ de $\mathcal V$} by $\pi$, by putting
$\mathcal V'=(M',E',N',\mathcal L')$, where $\mathcal L'$ is the transform of $\mathcal L$ by $\bar\pi$.

Proceeding as in the non-immersed case, we establish when a test system $\mathcal S$ over $(M,E)$ is {\em $\mathcal V$-permissible} and what is the {\em transform of $\mathcal V$ by a $\mathcal V$-permissible test system}.

We extend the definition of equivalence for idealistic spaces to the immersed case as follows:

\begin{definition}
Consider two immersed idealistic spaces $\mathcal V_\alpha$ and $\mathcal V_\beta$ over $(M,E)$ of respective dimensions $e_\alpha$ and $e_\beta$. We say that $\mathcal V_\alpha$ and $\mathcal V_\beta$ are {\em equivalent} if they have the same permissible test systems.
\end{definition}

\begin{remark}
Two equivalent immersed idealistic spaces
$$
\mathcal V_ \alpha=(M,E,N_\alpha, \mathcal L_\alpha), \quad \mathcal V_ \beta=(M,E,N_\beta, \mathcal L_\beta)
$$
have the same singular locus $S=\operatorname{Sing}\mathcal V_\alpha=\operatorname{Sing}\mathcal V_\beta$.
The subspaces $N_\alpha$ and $N_\beta$ are not necessarily equal, they can even have different dimensions, but $S\subset N_\alpha \cap N_\beta$. In the case when $N=N_\alpha=N_\beta$, we have that $\mathcal V_\alpha$ is equivalent to $\mathcal V_\beta$ if and only if the (non-immersed) idealistic spaces
$(N,E\vert_{N},\mathcal L_\alpha)$ and $(N,E\vert_{N},\mathcal L_\beta)$
are equivalent.
\end{remark}

\section{Idealistic Flowers}

Let us consider an ambient space $(M,E)$. An {\em immersed idealistic atlas $\mathcal P$ over $(M,E)$} is a finite family $\mathcal P=\{\mathcal V_\alpha\}_{\alpha\in \Lambda}$ such that:
\begin{enumerate}
\item The $\mathcal V_\alpha$ are immersed idealistic spaces over open sets $(M_\alpha,E_\alpha)$ of $(M,E)$, where $\{(M_\alpha,E_\alpha)\}_{\alpha \in \Lambda}$ is an open covering of $(M,E)$.
\item ({\em Compatibility property}) For any pair of indices $\alpha,\beta\in \Lambda$, the immersed idealistic spaces
$$
\mathcal V_{\alpha\beta}=\mathcal V_\alpha\vert_{M_{\alpha\beta}}, \quad
\mathcal V_{\beta\alpha}=\mathcal V_\beta\vert_{M_{\beta\alpha}}
$$
are equivalent, where $M_{\alpha\beta}=M_{\beta\alpha}=M_\alpha\cap M_\beta$.
\end{enumerate}
Each immersed idealistic space $\mathcal V_\alpha$ will be called an {\em immersed idealistic chart of $\mathcal P$}. When all the immersed idealistic charts $\mathcal V_\alpha$ of $\mathcal P$ have the same dimension $e$, we say that $\mathcal P$ is an {\em immersed idealistic $e$-atlas}.

An {\em immersed idealistic $e$-chart over $(M,E)$} is any $e$-dimensional immersed idealistic space $\mathcal C$ of the form $\mathcal C=(U,U\cap E, N,\mathcal L)$, where $U$ is an open subset of $\mathcal M$. We say that an immersed idealistic chart $\mathcal C$ is {\em compatible} with an immersed idealistic atlas $\mathcal P$ if $\mathcal P\cup \{\mathcal C\}$ is again an immersed idealistic atlas.

Let  $\mathcal P=\{\mathcal V_\alpha\}_{\alpha\in \Lambda}$ be an immersed idealistic atlas, where $$\mathcal V_\alpha=(M_\alpha,E_\alpha,N_\alpha,\mathcal L_\alpha).$$
Proceeding as in the non-immersed case (see Section \ref{ingl Atlasidealisticos}), we can define in a coherent way the following notions concerning $\mathcal P$:
\begin{itemize}
\item The singular locus $\operatorname{Sing}\mathcal P=\bigcup_{\alpha\in \Lambda} \operatorname{Sing}\mathcal V_\alpha$. It is a closed analytic subset of $M$, such that $M_\alpha\cap \operatorname{Sing}\mathcal P=\operatorname{Sing}\mathcal V_\alpha$.
\item The permissible centers $Y$ for $\mathcal P$. They are locally defined by the property that $Y\cap M_\alpha$ is permissible for $\mathcal V_\alpha$.
\item The $\mathcal P$-permissible test systems over $(M,E)$ and the transforms of $\mathcal P$ by $\mathcal P$-permissible test systems.
\end{itemize}

In the same way as for (non-immersed) idealistic atlases, we have the following result:

\begin{proposition} \label{prop:ingl equivalencia de atlas sumergidos}
Given two immersed idealistic atlases $\mathcal P_1$ and $\mathcal P_2$ over $(M,E)$, the following properties are equivalent:
\begin{enumerate}
		\item Any immersed idealistic chart of $\mathcal P_2$ is compatible with $\mathcal P_1$.
		\item Any immersed idealistic chart of $\mathcal P_1$ is compatible with $\mathcal P_2$.
		\item The union $\mathcal P_1\cup \mathcal P_2$ is an immersed idealistic atlas.
		\item $\mathcal P_1$ and $\mathcal P_2$ have the same permissible test systems.
	\end{enumerate}
\end{proposition}

\begin{definition}
Two immersed idealistic atlases $\mathcal P_1$ and $\mathcal P_2$ over $(M,E)$ are called to be {\em equivalent} if they have the same permissible test systems (and hence the equivalent properties in Proposition \ref{prop:ingl equivalencia de atlas sumergidos} hold).
\end{definition}

\begin{definition}
The equivalence classes of immersed idealistic atlases over $(M,E)$ are called  {\em idealistic flowers over $(M,E)$}.
\end{definition}

The properties and concepts invariant by the equivalence relation of immersed idealistic atlases define properties and concepts concerning idealistic flowers $\mathcal F$ over $(M,E)$. Thus, we have well defined:
\begin{itemize}
	\item Singular locus $\operatorname{Sing}\mathcal F$.
	\item Concept of permissible centers.
	\item Transformations by open projections.
	\item Transformations by blowing-ups of permissible centers.
	\item Concept of $\mathcal F$-permissible test systems.
	\item Transformations by $\mathcal F$-permissible test systems.
	\item Existence of reduction of singularities.
\end{itemize}

\subsection{Idealistic $e$-Flowers} The immersed idealistic atlases belonging to a given idealistic flower do not have necessarily a fixed dimension. Anyway, in order to prove the reduction of singularities by induction on the dimension, we need also to consider idealistic flowers of a fixed dimension. Thus, we take the following definition:

\begin{definition} Let $(M,E)$ be an ambient space of dimension $n$. Consider an integer number $0\leq e\leq n$. An {\em idealistic $e$-flower over $(M,E)$} is an equivalence class of
immersed idealistic $e$-atlases over $(M,E)$.
\end{definition}
Let us remark that the equivalence relation of immersed idealistic atlases is the same one in both cases, but it concerns two different sets:
$$
\left\{
\begin{array}{c}
\text{immersed idealistic $e$-atlases}\\
\text{over } (M,E)
\end{array}
\right\}
\subset
\left\{
\begin{array}{c}
\text{immersed idealistic atlases}\\
\text{over }(M,E)
\end{array}
\right\}.
$$

Let $\mathcal F_1$ and $\mathcal F_2$ be idealistic flowers of respective dimensions $e_1$ and $e_2$ over the same ambient space $(M,E)$. We say that $\mathcal F_1$ and $\mathcal F_2$  are {\em equivalent} if they are contained in an common idealistic flower $\mathcal F$. This is the same to say that any atlas of $\mathcal F_1$ is equivalent to any other of $\mathcal F_2$. It is also the same that asking $\mathcal F_1$ and $\mathcal F_2$ to have the same permissible test systems.  Note that, when $e_1=e_2$, then $\mathcal F_1$ and $\mathcal F_2$ are equal or disjoint. In the particular case that $e_2=e_1-1$, we say that $\mathcal F_2$ has {\em maximal contact} with $\mathcal F_1$.

\subsection{Description in Terms of Idealistic Exponents} Let $(M,E)$ be an ambient space. An {\em immersed idealistic exponent} over $(M,E)$ is a data
$$
\mathcal {W}=(M,E,N,\mathcal E),
$$
where $(M,E,N)$ is a transverse closed ambient subspace of $(M,E)$ and $\mathcal E$ is an idealistic exponent over $(N,E\vert_N)$. The {\em singular locus $\operatorname{Sing}\mathcal {W}$} is $\operatorname{Sing}\mathcal {W}=\operatorname{Sing}\mathcal {E}$. The {\em dimension} of $\mathcal W$ is the dimension of $N$.

The immersed idealistic exponents are well-transformed by restriction to open sets, projections over the first factor and permissible blowing-ups, where the permissible centers are, by definition, the permissible centers for $\mathcal E$. Thus, the permissible test systems and the corresponding transforms are also defined. Two immersed idealistic exponents are {\em equivalent} if they have the same permissible test systems.

\begin{remark}
Let us consider two immersed idealistic exponents
 $$
\mathcal {W}_\alpha=(M,E,N_\alpha,\mathcal E_\alpha)
,\quad
\mathcal {W}_\beta=(M,E,N_\beta,\mathcal E_\beta).
$$
If $\mathcal {W}_\alpha$ and $\mathcal {W}_\beta$ are equivalent and $N_\alpha=N_\beta$, then $\mathcal E_\alpha=\mathcal E_\beta$. Nevertheless, we can have that $\mathcal{W}_\alpha$ and  $\mathcal {W}_\beta$ are equivalent with $N_\alpha\ne N_\beta$. In this case, we can assure that
$\operatorname{Sing}(\mathcal{W}_\alpha)=\operatorname{Sing}(\mathcal{W}_\beta)\subset N_\alpha\cap N_\beta$.
\end{remark}

An {\em immersed exp-idealistic atlas} is a finite family
$$
\mathcal Q=\{\mathcal {W}_\alpha\}_{\alpha\in \Lambda}, \quad \mathcal {W}_\alpha=(M_\alpha,E_\alpha,N_\alpha,\mathcal E_\alpha),
$$
where the $\mathcal {W}_\alpha$ are {\em immersed idealistic exponents} over $(M_\alpha,E_\alpha)$, the family $\{(M_\alpha,E_\alpha)\}_{\alpha \in \Lambda}$ is an open covering of $(M,E)$ and we have the usual compatibility condition among the $\mathcal {W}_\alpha$. When all the charts $\mathcal W_\alpha$ have dimension equal to $e$, we say that $\mathcal Q$ {\em has dimension $e$}; we also say that
 $\mathcal Q$ is an {\em immersed exp-idealistic $e$-atlas}.

The singular locus, permissible centers, permissible test systems and transforms by a permissible test systems are defined as usual. Two  immersed exp-idealistic atlases over $(M,E)$ are {\em equivalent} if they have the same permissible test systems.

\begin{proposition}
\label{porp:ingl apilamientosenterminosdeexponentes}
  Let $(M,E)$ be an ambient space. There is a natural bijection between equivalence classes of immersed exp-idealistic atlases over $(M,E)$ and idealistic flowers over $(M,E)$. In the same way, there is a natural bijection between equivalence classes of immersed exp-idealistic $e$-atlases over $(M,E)$ and idealistic $e$-flowers over $(M,E)$.
\end{proposition}
\begin{proof}
Let us consider the map
$$
\Psi: \left\{
\begin{array}{c}
\text{immersed idealistic }\\
\text{atlases over } (M,E)
\end{array}
\right\}
\to
\left\{
\begin{array}{c}
\text{immersed exp-idealistic }\\
\text{atlases over }(M,E)
\end{array}
\right\},
$$
defined as follows. Take an immersed idealistic atlas $\mathcal P=\{\mathcal V_\alpha\}_{\alpha \in \Lambda}$ over $(M,E)$, where
$$
\mathcal V_\alpha=(M_\alpha,E_\alpha,N_\alpha,\mathcal L_\alpha).
$$
The idealistic space $(N_\alpha,E_\alpha|_{N_\alpha},\mathcal L_\alpha)$ defines an idealistic exponent $\mathcal E_\alpha$ over $(N_\alpha,E_\alpha|_{N_\alpha})$. We take $\Psi(\mathcal P)=\{\mathcal W_\alpha\}_{\alpha \in \Lambda}$, where
$$
\mathcal W_\alpha=(M_\alpha,E_\alpha,N_\alpha,\mathcal E_\alpha).
$$
The map $\Psi$ is well-defined and it induces a bijection up to equivalence. The case when the dimension is fixed in done in the same way.
\end{proof}
\subsection{Change of Codimension}
\label{subsec: change of codimension}
Let $(M,E)$ be an ambient space and $(M,E,N)$ a closed transverse subspace. Consider two immersed idealistic charts
$$
\mathcal V_A=(N,E|_N,A,\mathcal L_A), \quad \mathcal V_B=(N,E|_N,B,\mathcal L_B)
$$
over the ambient space $(N,E|_N)$. We can consider two new immersed idealistic charts
$$
\mathcal V^*_A=(M,E,A,\mathcal L_A), \quad \mathcal V^*_B=(M,E,B,\mathcal L_B)
$$
over the ambient space $(M,E)$. We have the following property:
\begin{proposition}
The immersed idealistic charts $\mathcal V_A$ and $\mathcal V_B$ are equivalent if and only if $\mathcal V^*_A$ and $\mathcal V^*_B$ are also equivalent.
\end{proposition}
We have the same result if we consider immersed exp-idealistic charts
$$
\mathcal W_A=(N,E|_N,A,\mathcal E_A), \quad \mathcal W_B=(N,E|_N,B,\mathcal E_B)
$$
and
$\mathcal W^*_A=(M,E,A,\mathcal E_A)$, $\mathcal W^*_B=(M,E,B,\mathcal E_B)$.

\section{Reduction of Singularities} Consider an ambient space $(M,E)$ and let $\mathcal O$ be an object over $(M,E)$   belonging to one of the following classes:
\begin{itemize}
  \item Idealistic Spaces over $(M,E)$.
  \item Idealistic Atlases over $(M,E)$.
  \item Idealistic Exponents over $(M,E)$.
  \item Immersed Idealistic $e$-Spaces over $(M,E)$.
  \item Immersed Idealistic $e$-Atlases over $(M,E)$.
  \item Idealistic $e$-Flowers over $(M,E)$.
\end{itemize}
We say that there is a {\em reduction of singularities for  $\mathcal O$} if and only if there is a morphism $\pi:(M',E')\rightarrow (M,E)$ that is a finite composition of permissible blowing-ups such that the transform $\mathcal O'$ of $\mathcal O$ has empty singular locus. In view of the description of the objects and the transformations, the object $\mathcal O$ defines in a proper way an idealistic $e$-flower $\mathcal F_\mathcal O$ and any reduction of singularities of $\mathcal F_\mathcal O$ induces a reduction of singularities of $\mathcal O$ in its own class.

Then, the main result of this notes is the following one:
\begin{theorem}
\label{teo:ingl redsingflowers}
  Let $\mathcal F$ be an idealistic $e$-flower over an ambient space $(M,E)$. Then, there is a reduction of singularities for $\mathcal F$.
\end{theorem}

We will provide a proof of this result, working essentially by induction on the dimension $e$ of the idealistic flower $\mathcal F$.

\vfill
\pagebreak
\part{Hironaka's Order}
\label{The order of Idealistic Exponents and Equidimensional Flowers}
In this part we develop the definition and properties of the {\em order} of an idealistic $e$-flower. We do that for idealistic spaces, idealistic atlases, immersed idealistic $e$-atlases and finally for idealistic $e$-flowers.

The concept of ``order'' is important to define the so-called ``adjusted'' idealistic $e$-flowers. The most important step in the proof of Theorem \ref{teo:ingl redsingflowers} consists in proving the existence of reduction of singularities for adjusted (and reduced) $e$-flowers under the induction hypothesis that there is reduction of singularities for any $(e-1)$-flower. This process is founded in the Maximal Contact Theory \cite{Aro-H-V}, as we shall see in next parts.

\section{Order of an Idealistic Space}

Let $\mathcal M=(M,E, \mathcal L)$ be an idealistic space, where $\mathcal L=\{(I_j,d_j)\}_{i=1}^k$. Consider a point $P\in M$. Following Hironaka \cite{Aro-H-V}, the {\em order $\delta_P\mathcal L$} is defined by
$$
\delta_P\mathcal L=\min\left\{\frac{\nu_PI_j}{d_j};\quad j=1,2,\ldots,k\right\}\in \mathbb Q_{\geq 0}.
$$
Note that $P$ is a singular point if and only if $\delta_P\mathcal L\geq 1$. We also use the notation $\delta_P\mathcal M=\delta_P\mathcal L$, although the divisor $E$ has no relevance in the definition of the order. In a similar way, we define the {\em generic order  $\delta_Y\mathcal L$} along an irreducible closed subspace $Y$ of $M$.

\begin{proposition}
\label{prop:ingl explosionexpcto}
Consider an idealistic space $\mathcal M=(M,E,\mathcal L)$ and a permissible irreducible center $Y\subset M$. Let $\pi:(M',E')\rightarrow (M,E)$ be the blowing-up with center $Y$ and let $\mathcal M'$ be the transform of $\mathcal M$ by $\pi$. We have that
 	\begin{equation} \label{eq:ingl expctodivexcep}
 		\delta_{D'}\mathcal M'=\delta_Y\mathcal M-1,
 	\end{equation}
 	where $D'=\pi^{-1}(Y)$ is the exceptional divisor of $\pi$.
 \end{proposition}
 \begin{proof}
 	Write $\mathcal L=\{(I_j,d_j)\}_{j=1}^k$ and let $\mathcal J_{D'}$ the ideal sheaf defining the exceptional divisor $D'$. Recall that the transformed list $\mathcal L'=\{(I'_j,d_j)\}_{j=1}^k$ is given by
 	$$
 	I'_j=\mathcal J_{D'}^{-d_j}\pi^{-1}(I_j), \quad j=1,2,\ldots,k.
 	$$
 	Put $m_j=\nu_Y(I_j)$. The strict transform $I^*_j \subset \mathcal O_{M'}$ satisfies that
 	$$
 	\pi^{-1}(I_j)= \mathcal J_{D'}^{m_j}I^*_j,\quad \nu_{D'}(I^*_j)=0.
 	$$
 	We have $\nu_{D'}(\pi^{-1}(I_j))=m_j$, for all $j=1,2,\ldots,k$.
 	We conclude that $\nu_{D'}(I'_j)= \nu_Y(I_j)-d_j$. Hence, we have
 	$$
 	\nu_{D'}(I'_j)/d_j= (\nu_Y(I_j)/d_j)-1, \quad j=1,2,\ldots,k.
 	$$
 	Taking the minimal values, we get the equality $\delta_{D'}\mathcal M'=\delta_Y\mathcal M-1$.
 \end{proof}
 \subsection{Curve-Divisor Situation}
 \label{Curve-Divisor Situation}
 In this subsection, we introduce a cons\-truction, usually called ``Hironaka's trick'', that is useful for proving that the order is an invariant under the equivalence of idealistic spaces.

 A {\em curve-divisor situation} $(\mathcal M,X,D,P)$ is given by:
 \begin{enumerate}
 	\item An idealistic space $\mathcal M=(M,E,\mathcal L)$.
 	\item A non-singular closed irreducible curve $X\subset\operatorname{Sing}\mathcal M$.
 	\item A non-singular closed irreducible hypersurface $D\subset M$ having normal crossings with $E$ and with $X$, such that $P=X\cap D$ (note that $P$ is a singular point).
 \end{enumerate}
Consider a curve-divisor situation $(\mathcal M,X,D,P)$. We are going to define an $\mathcal M$-permissible infinite test system $\mathcal S_{\mathcal M,X,D,P}$ associated to the curve-divisor situation.

Write
$
\mathcal S_{\mathcal M,X,D,P}= \{(Y_{j-1},\sigma_{j})\}_{j=1}^\infty.
$
The morphisms $\sigma_j$ are only of blowing-up type and we give inductively the centers $Y_{j-1}$ as follows. If $j=1$, we put
 \begin{enumerate}
 	\item $Y_0=D$, if $D$ is a permissible center for $\mathcal M$.
 	\item $Y_0=\{P\}$, otherwise.
 \end{enumerate}
 We respectively denote by $\mathcal M'$, $X'$, $D'$ the transform of $\mathcal M$, the strict transform of $X$ and the exceptional divisor with respect to $\sigma_1$. Since $\delta_{X'}\mathcal M'=\delta_X\mathcal M$, the curve $X'$ is in the singular locus of $\mathcal M'$, hence we get a new curve-divisor situation $(\mathcal M',X',D',P')$, where $P'=X'\cap D'$. We proceed indefinitely in this way.

 \begin{lemma} \label{lema:ingl vaiacionexpcto2}
 	We have the following properties:
 	\begin{enumerate}
 		\item $\delta_P\mathcal M\geq \delta_X\mathcal M+\delta_D\mathcal M$.
 		\item If $\delta_P\mathcal M= \delta_X\mathcal M+\delta_D\mathcal M$, then $\delta_{P'}\mathcal M'= \delta_{X'}\mathcal M'+\delta_{D'}\mathcal M'$.
 	\end{enumerate}
 \end{lemma}
 \begin{proof}
 	Let  $(I,d)$ be a marked ideal of $\mathcal L$ with $\delta_P\mathcal L=\nu_PI/d$. Let us denote $m=\nu_DI$ and write $I=\mathcal J_D^mJ$, where $\mathcal J_D$ defines $D$; hence $\nu_DJ=0$. Note that $m/d\geq\delta_D\mathcal L$. On the other hand, we have
 	$$
 	\delta_P\mathcal L=\nu_PI/d=\nu_PJ/d+ m/d\geq \nu_PJ/d+\delta_D\mathcal L.
 	$$
 	Since  $X\not\subset D$ we have that $\nu_X(\mathcal J_D)=0$ and $\nu_XI=\nu_XJ$. Moreover, since $P\in X$, we see that $\nu_PJ\geq \nu_XJ=\nu_XI$. Then, we conclude
 	$$
 	\delta_P\mathcal L=\nu_PI/d=\nu_PJ/d +m/d\geq \nu_XI/d+\delta_D\mathcal L\geq\delta_X\mathcal L+ \delta_D\mathcal L.
 	$$
 	This shows (1).
 	Let us see (2). If the center of $\sigma_1$ is $D$, the morphism $\sigma_1$ is the identity and we have
 	$$
 	\delta_{P'}\mathcal L'=\delta_{P}\mathcal L-1,\quad  \delta_{X'}\mathcal L'=\delta_X\mathcal L,\quad \delta_{D'}\mathcal L'=\delta_{D}\mathcal L-1.
 	$$
 	Now, property (2) is straightforward. Assume now that  $\sigma_1$ is the blowing-up with center $P$. Let us consider the list $\widetilde{\mathcal L}=\{(I,d)\}\subset \mathcal L$ given by the only marked ideal $(I,d)$. We know that
 	$$
 	\delta_P \widetilde{\mathcal L}=\delta_P{\mathcal L},
 	\quad
 	\delta_X \widetilde{\mathcal L}\geq \delta_X {\mathcal L},
 	\quad
 	\delta_D \widetilde{\mathcal L}\geq \delta_D {\mathcal L}.
 	$$
 	By hypothesis, the equality   $\delta_P\mathcal L= \delta_X\mathcal L+\delta_D\mathcal L$ holds. Then, we have that $\delta_P\widetilde{\mathcal L}\leq
 	\delta_X\widetilde{\mathcal L}+\delta_D\widetilde{\mathcal L}$. Applying Statement (1), it follows that
 	$$
 	\delta_P\widetilde{\mathcal L}=
 	\delta_X\widetilde{\mathcal L}+\delta_D\widetilde{\mathcal L},\quad
 	\delta_X\widetilde{\mathcal L}=\delta_X{\mathcal L},\quad \delta_D\widetilde{\mathcal L}=\delta_D{\mathcal L}.
 	$$
 	Now, assume that the following equality holds
 	\begin{equation} \label{eq:ingl reducionaunideal}
 		\delta_{P'}\widetilde{\mathcal L}'=
 		\delta_{X'}\widetilde{\mathcal L}'+\delta_{D'}\widetilde{\mathcal L}'.
 	\end{equation}
 	Let us see how to end the proof of Statement (2). Recall that
 	$$
 	\delta_{D'}\widetilde{\mathcal L}'=\delta_P\widetilde{\mathcal L}-1,\;
 	\delta_{D'}{\mathcal L}'=\delta_P{\mathcal L}-1,\;
 	\delta_{X'}\mathcal L'=\delta_{X}\mathcal L,\;
 	\delta_{X'}\widetilde{\mathcal L}'=\delta_X\widetilde{\mathcal L}.
 	$$
 	Moreover, since $\widetilde{\mathcal L}'\subset\mathcal L'$, we have that $\delta_{P'}\mathcal L'\leq \delta_{P'}\widetilde{\mathcal L}'$. We conclude that
 	\begin{eqnarray*}
 		\delta_{P'}\mathcal L'&\leq& \delta_{P'}\widetilde{\mathcal L}'=
 		\delta_{X'}\widetilde{\mathcal L}'+\delta_{D'}\widetilde{\mathcal L}'=\\
 		&=&\delta_X\widetilde{\mathcal L}+\delta_P\widetilde{\mathcal L}-1=\\
 		&=& \delta_{X}{\mathcal L}+\delta_P{\mathcal L}-1=
 		\delta_{X'}{\mathcal L}'+\delta_{D'}{\mathcal L}'.
 	\end{eqnarray*}
 	By Statement (1), we have $\delta_{P'}\mathcal L'=\delta_{X'}{\mathcal L}'+\delta_{D'}{\mathcal L}'$.
 	
 	Now, it is enough to show the equality in Equation \eqref{eq:ingl reducionaunideal}.
 	Following the proof of Statement (1), the fact  $\delta_P\widetilde{\mathcal L}=
 	\delta_X\widetilde{\mathcal L}+\delta_D\widetilde{\mathcal L}$ implies that $\nu_XJ=\nu_PJ$. That is, the ideal $J$ is equimultiple along $X$ around $P$.
 	Denote by $J'$ the strict transform of  $J$  by $\sigma_1$. We have that $\nu_{P'}J'\leq \nu_PJ$. Since $P'\in X'$, we also have that $\nu_{P'}J'\geq \nu_{X'}J'$. Combining these properties with the facts that $\nu_{X'}J'=\nu_XJ$ and $\nu_PJ=\nu_XJ$, we conclude that
 	$$
 	\nu_{P'}J'=\nu_PJ=\nu_XJ=\nu_{X'}J'.
 	$$
 	Recall that the strict transform of $I$ coincides with the strict transform of $J$, at the point $P'$, since the strict transform of $D$ does not go through $P'$. Then, we can write the controlled transform  $I'$ of $I$ at $P'$ as $I'=\mathcal J_{D'}^{m'}J'$, where $m'=\nu_{D'}I'$. This implies that
 	$$
 	d\cdot\delta_{X'}\widetilde{\mathcal L}'=\nu_{X'}I'=\nu_{X'}J'+m'\cdot\nu_{X'}\mathcal J_{D'}=\nu_{X'}J'=\nu_{P'}J',
 	$$
 	since $\nu_{X'}\mathcal J_{D'}=0$, in view of the fact that $X'$ is not contained in $D'$. Finally, we have that
 	$$
 	\delta_{P'}\widetilde{\mathcal L}'=\nu_{P'}I'/d=\nu_{P'}J'/d+\nu_{D'}I'/d=\delta_{X'}\widetilde{\mathcal L}'+\delta_{D'}\widetilde{\mathcal L}'.
 	$$
 	This ends the proof.
 \end{proof}
\begin{remark}
\label{obs:ingl coincidenciasistcurvadivisor}
Assume that $\mathcal M_\alpha$ y $\mathcal M_\beta$ are two equivalent idealistic spaces over $(M,E)$. Assume that $(\mathcal M_\alpha,X,D,P)$ is a curve-divisor situation. Then $(\mathcal M_\beta,X,D,P)$ is also a curve-divisor situation and we have that the two associated infinite test systems coincide, that is
 	$$
 	\mathcal S_{\mathcal M_\alpha,X,D,P}=\mathcal S_{\mathcal M_\beta,X,D,P}.
 	$$
 \end{remark}

 \subsection{Invariance under Equivalence}
 \label{Invariance under Equivalence of Idealistic Spaces}
 In this subsection we give a proof of the following statement:

 \begin{proposition} \label{prop:ingl expctoexpideal}
 	Let $\mathcal M_\alpha$ and $\mathcal M_\beta$ be two equivalent idealistic spaces over the ambient $(M,E)$. Let $S$ be their common singular locus. For any $P\in S$, we have that $\delta_P\mathcal M_\alpha=\delta_P\mathcal M_\beta$.
 \end{proposition}

 The proof of Proposition \ref{prop:ingl expctoexpideal} is based in the construction of a particular curve-divisor situation, obtained after performing a projection over the first factor, as follows.

 Take an idealistic space $\mathcal M$ over $(M,E)$ and a point $P\in \operatorname{Sing}\mathcal M$. Let $\sigma_1$ be the projection over the first factor
 $$
 \sigma_1:(M_1,E_1)=(M\times(\mathbb C,0),E\times (\mathbb C,0))\rightarrow (M,E).
 $$
 Denote by $\mathcal M_1$ the transform of $\mathcal M$ by $\sigma_1$ and let us put
 $$
 X_1=\{P\}\times (\mathbb C,0),\;
 D_1=M\times(\mathbb C,0),\;
 P_1=(P, 0).
 $$
We have a curve-divisor situation $(\mathcal M_1,X_1,D_1,P_1)$. Note that
\begin{equation*} \label{eq:ingl equisuno}
 \delta_{X_1}\mathcal M_1=\delta_{P_1}\mathcal M_{1}=\delta_P\mathcal M\geq 1, \quad \delta_{D_1}\mathcal M_1=0.
\end{equation*}
 In particular $\delta_{P_1}\mathcal M_{1}=\delta_{X_1}\mathcal M_{1}+\delta_{D_1}\mathcal M_1$.

 Consider the test system
 $\mathcal S_{\mathcal M_1,X_1,D_1,P_1}=\{(Y_{j-1},\sigma_j)\}_{j=2}^\infty
 $
  and denote by $(\mathcal M_j,X_j,D_j,P_j)$ the transformed curve-divisor situations, for $j\geq 2$.

 Let us make a computation of orders, by an extensive application of Proposition \ref{prop:ingl explosionexpcto} and Lemma \ref{lema:ingl vaiacionexpcto2}. In order to simplify notations, we write $e=\delta_P\mathcal M=\delta_{X_1}\mathcal M_1$ and
 $$
 a_j=\delta_{D_j}\mathcal M_{j}, \quad
 e_j=\delta_{P_j}\mathcal M_j,\quad  j\geq 1.
 $$
The first remark is that $e=e_1=\delta_{X_1}\mathcal M_1=\delta_{X_j}\mathcal M_j$, for all $j\geq 2$. We also know that  $a_1=0$. By Lemma \ref{lema:ingl vaiacionexpcto2}, we get the relations
 \begin{equation} \label{eq:ingl hipcruzamientos}
 	e_j=e+a_j, \quad j\geq 1.
 \end{equation}
 We are going to describe the sequence $\{a_j\}_{j=2}^\infty$ recurrently from the rational number $e\geq 1$. If $j=2$, by Proposition \ref{prop:ingl explosionexpcto}, we have that
 $$
 a_2=e-1,
 $$
 Assume $j> 2$. If $a_j\geq 1$, then $D_j$ is contained in the singular locus of $\mathcal M_j$ and we have $Y_{j}=D_j$. In this case, we see that
 \begin{equation}
 	\label{eq:ingl variacioncto1}
 	a_{j+1}=a_j-1.
 \end{equation}
 If $a_j<1$, the divisor $D_j$ is not in the singular locus and $Y_j=\{P_j\}$. By Proposition  \ref{prop:ingl explosionexpcto}, we have  $a_{j+1}=e_j-1$; in view of Equation \eqref{eq:ingl hipcruzamientos}, we get
 \begin{equation}
 	\label{eq:ingl variacioncto2}
 	a_{j+1}=a_j+e-1.
 \end{equation}
 Thus, the sequence $\{a_j\}_{j=2}^\infty$ is inductively obtained from $e$ as follows:
 \begin{enumerate}[a)]
 	\item $a_2=e-1$.
 	\item If $a_j\geq 1$, then $a_{j+1}=a_j-1$.
 	\item If $a_j<1$, then $a_{j+1}=a_j+e-1$
 \end{enumerate}

 Since the values $a_j \geq 0$ are rational numbers having a common denominator, a diophantine computation shows that there is a first index $j_0\geq 2$ such that $a_{j_0}=0$.

Now, we can conclude the proof of Proposition \ref{prop:ingl expctoexpideal}. Since $\mathcal M_\alpha$ and $\mathcal M_\beta$ are equivalent, we have that $\mathcal M_{\alpha,1}$ is equivalent to $\mathcal M_{\beta,1}$ and hence the infinite test systems $\mathcal S_{\mathcal M_{\alpha,1},X_1,D_1,P_1}$ and $\mathcal S_{\mathcal M_{\beta,1},X_1,D_1,P_1}$ coincide, in view of Remark \ref{obs:ingl coincidenciasistcurvadivisor}. In particular, we have that
 $$
 a_j^\alpha\geq 1\Leftrightarrow a_j^\beta\geq 1,\quad j\geq 2.
 $$
 In other words, the construction of the sequences $\{a^\alpha_j\}$ and $\{a^\beta_j\}$ follows the same steps, starting with $e^\alpha=\delta_P\mathcal M_\alpha$ and $e^\beta=\delta_P\mathcal M_\beta$, respectively. We have to prove that $e^\alpha=e^\beta$. Assume by contradiction that $e^\alpha>e^\beta$. We have that $a^\alpha_j>a^\beta_j$, for all $j\geq 2$. Now, choosing $j_0\geq 2$ such that $a^\alpha_{j_0}=0$, we conclude that
 $a^\beta_{j_0}<0$, which is impossible.

\section{Order for Idealistic Exponents and $e$-Flowers}
Let $\mathcal A=\{\mathcal M_\alpha\}_{\alpha\in \Lambda}$ be an idealistic atlas over $(M,E)$. Given a point  $P\in \operatorname{Sing}\mathcal A$, we define the {\em order $\delta_P\mathcal A$} by
$$
\delta_P\mathcal A=\delta_P\mathcal M_\alpha,
$$
where $(M_\alpha,E_\alpha)$ is the ambient space of an idealistic chart $\mathcal M_\alpha$ such that $P\in M_\alpha$. In view of the Proposition \ref{prop:ingl expctoexpideal}, the definition does not depend  on the particular idealistic chart $\mathcal M_\alpha$ such that $P\in M_\alpha$. On the other hand, since $P\in \operatorname{Sing}(\mathcal M_\alpha)$, we have that $\delta_P\mathcal A\geq 1$.

Given an idealistic exponent $\mathcal E$ over $(M,E)$ and a point $P\in \operatorname{Sing}\mathcal E$, the {\em order $\delta_P\mathcal E$} is $\delta_P\mathcal E=\delta_P\mathcal A$, where $\mathcal A$ is any atlas defining $\mathcal E$.

Consider now an immersed idealistic $e$-space $\mathcal V=(M,E,N,\mathcal L)$ and a point $P\in \operatorname{Sing}\mathcal V$. We define the {\em order $\delta_P\mathcal V$} to be
$$
\delta_P\mathcal V=\delta_P\mathcal N,\quad \mathcal N=(N,E\vert_N,\mathcal L).
$$
The reader can verify that the arguments in Subsections \ref{Curve-Divisor Situation} and  \ref{Order Invariance under Equivalence of Idealistic Spaces} work in a similar way for immersed idealistic spaces, with fixed dimension $e$. More precisely, we have:
\begin{proposition}
\label{prop:ingl ordensumergido}
Let $\mathcal V_\alpha$ y $\mathcal V_\beta$ be two equivalent immersed idealistic spaces over $(M,E)$ of the same dimension $e$. For any point $P$ in the common singular locus
$\operatorname{Sing}\mathcal V_\alpha=\operatorname{Sing}\mathcal V_\beta$, we have that
$\delta_P\mathcal V_\alpha=\delta_P\mathcal V_\beta$.
\end{proposition}

\begin{remark} An example that the equality of the dimension is necessary in the statement of Proposition \ref{prop:ingl ordensumergido} is the following one. Take   $\mathcal V_\alpha$ and  $\mathcal V_\beta$ over  $(M,E)=((\mathbb C^2,0), \emptyset)$, given by
	$N_\alpha=M$, $N_{\beta}=(y=0)$, with  $\mathcal L_\alpha=\{\mathcal I_\alpha\}$, $\mathcal L_\beta=\{\mathcal I_\beta\}$, where:
	$$
	\mathcal I_\alpha=\left((y^2-x^3)\mathcal O_{\mathbb C^2,0},2\right),\quad \mathcal I_\beta= \left((x^3\mathcal O_{\mathbb C,0},2)\right).
	$$
By Maximal Contact Theory \cite{Aro-H-V}, we know that $\mathcal V_\alpha$ and $\mathcal V_\beta$ are equivalent. Nevertheless, we have that $\delta_0\mathcal V_\alpha=1$ and $\delta_0\mathcal V_\beta=3/2$.
\end{remark}

Thus, we can define the {\em order for immersed idealistic $e$-atlases and for idealistic $e$-flowers}, but this is not possible for general immersed idealistic atlases and  idealistic flowers.
\\
\strut\vfill\pagebreak
\part{Guide for the Reduction of Singularities}
\label{Guide for the Reduction of Singularities}
Here we give a quick guide for the proof of the existence of reduction of singularities for idealistic $e$-flowers. The original result (Theorem \ref{teo:Ingl redsing}) we want to prove is the following one:

\begin{theorem}
\label{th: redsingidealisticspaces} Given an ambient space $(M,E)$,
there is a reduction of singularities for any idealistic space $\mathcal M$ over $(M,E)$.
\end{theorem}

This result is a consequence of the following one:
\begin{theorem}
\label{th: redsingidealisticatlas} Given an ambient space $(M,E)$,
there is a reduction of singularities for any idealistic atlas  $\mathcal A$ over $(M,E)$.
\end{theorem}

Now, Theorem \ref{th: redsingidealisticatlas} is equivalent to the next one:
\begin{theorem}
\label{th: redsingidealisticexponents} Given an ambient space $(M,E)$,
there is a reduction of singularities for any idealistic exponent  $\mathcal E$ over $(M,E)$.
\end{theorem}
We also have the immersed statements:

\begin{theorem}
\label{th: redsingimmersedidealisticspaces} Given an ambient space $(M,E)$ and a natural number $e$ with $1\leq e\leq \dim M$,
there is a reduction of singularities for any immersed idealistic $e$-space $\mathcal V$ over $(M,E)$.
\end{theorem}
\begin{theorem}
\label{th: redsingimmersedidealisticatlas} Given an ambient space $(M,E)$ and a natural number $e$ with $1\leq e\leq \dim M$,
there is a reduction of singularities for any immersed idealistic $e$-atlas $\mathcal P$ over $(M,E)$.
\end{theorem}
\begin{theorem} (see Theorem \ref{teo:ingl redsingflowers})
\label{th: redsingimmersedidealisticflowers} Given an ambient space $(M,E)$ and a natural number $e$ with $1\leq e\leq \dim M$,
there is a reduction of singularities for any idealistic $e$-flower $\mathcal F$ over $(M,E)$.
\end{theorem}
In view of the developments of the concepts in Part \ref{Objects and Statements}, we have the following implications:
$$
\begin{array}{ccccc}
\text{Th\ref{th: redsingimmersedidealisticflowers}}&\Leftrightarrow&
\text{Th\ref{th: redsingimmersedidealisticatlas}}&\Rightarrow&
\text{Th\ref{th: redsingimmersedidealisticspaces}}\\
\Downarrow&&\Downarrow&&\Downarrow\\
\text{Th\ref{th: redsingidealisticexponents}}&\Leftrightarrow&
\text{Th\ref{th: redsingidealisticatlas}}&\Rightarrow&
\text{Th\ref{th: redsingidealisticspaces}}\\
\end{array}
$$
We have to prove Theorem \ref{th: redsingimmersedidealisticflowers} or, equivalently, Theorem \ref{th: redsingimmersedidealisticatlas}.

Thanks to the behavior of the order presented in Part
\ref{The order of Idealistic Exponents and Equidimensional Flowers}, we can define {\em adjusted idealistic $e$-flowers} to be the ones such that the order in all the singular points is exactly equal to one. We also consider {\em reduced idealistic $e$-flowers}, given by the property that the singular locus has dimension less than or equal to $e-2$. In our general procedure, we need as well the concept of {\em monomial idealistic $e$-flower} to be developed further, that implies the existence of a very ``combinatorial'' reduction of singularities.

Let us see how to organize the proof of Theorem \ref{th: redsingimmersedidealisticflowers}. Consider the following statements:

\begin{description}
\item[\rm RedSing($e$)] {\em The idealistic $e$-flowers have reduction of singularities.}
\item [\rm RedSing(monomial)]{\em  The monomial idealistic $t$-flowers have reduction of singularities, for any $t\geq 1$.}
\item[\rm RedSing($e$, adjusted)] {\em The adjusted idealistic $e$-flowers have reduction of singularities.}
\item[\rm RedSing($e$, adjusted-reduced)]{\em The adjusted and reduced idealistic $e$-flowers have reduction of singularities.}
\end{description}
The verification that RedSing($1$) holds is straightforward. Our objective is to prove the induction step
\begin{equation}
\label{eq: inductionstep}
\text{\rm RedSing($e-1$)}\Rightarrow \text{\rm RedSing($e$)}.
\end{equation}
We do it in several steps:
\begin{description}
\item[Monomial case] \rm RedSing(monomial) holds.
\item[Reduction to the Adjusted Case] This statement corresponds to the following implication
    \begin{equation}
\label{eq:ingl reduccion al caso ajustado}
\left.
\begin{array}{l}
\text{RedSing($e$, adjusted)}\\\text{RedSing(monomial)}
\end{array}
\right\}
\Rightarrow \text{RedSing($e$)}.
\end{equation}
\item[Reduction to the Adjusted-Reduced Case] This statement corresponds to the following implication
    \begin{equation}
\label{eq:ingl reduccion al caso reducido}
\left.
\begin{array}{l}
\text{RedSing($e-1$)}\\
\text{RedSing($e$, adjusted-reduced)}
\end{array}
\right\}
\Rightarrow \text{RedSing($e$, adjusted)}.
\end{equation}

\item[The Adjusted-Reduced Case] This statement corresponds to the following implication
\begin{equation}
\label{eq:ingl contacto maximal}
\text{RedSing($e-1$)}\Rightarrow \text{RedSing($e$, adjusted-reduced)}.
\end{equation}
\end{description}
The statement \eqref{eq:ingl contacto maximal} is proven in terms of Maximal Contact Theory \cite{Aro-H-V}. The classical difficulty for the globalization of the maximal contact is the reason for introducing idealistic flowers, instead of working simply with idealistic exponents.
\\
\strut\vfill\pagebreak

\part{First Reductions}
In this part, we prove the statements ``Monomial Case'', ``Reduction to the Adjusted Case'' and ``Reduction to the Adjusted-Reduced Case'' presented in
Part \ref{Guide for the Reduction of Singularities}.

\section{Monomial Idealistic $e$-Flowers}
Let us consider an ambient space $(M,E)$. We say that a sheaf $Z$ of principal ideals on $M$ is {\em logarithmic for $(M,E)$} if it is of the form
\begin{equation}
\label{eq:ingl hazlogaritmico}
\textstyle Z=\prod_D\mathcal J_D^{a_D},\quad a_D\in {\mathbb Z}_{\geq 0},
\end{equation}
where $D$ run over  the irreducible components of $E$ and  $\mathcal J_D$ is the ideal sheaf of $D$. A {\em logarithmic idealistic space with assigned multiplicity $d$ over $(M,E)$} is any idealistic space $\mathcal Z$ of the form
$$
\mathcal Z=(M,E,\mathcal L_\mathcal Z=\{(Z,d)\}),
$$
where $Z$ is a logarithmic sheaf for $(M,E)$.

Let $(M,E,N)$ be a transverse ambient subspace of $(M,E)$ and let $Z$ be a logarithmic sheaf. Note that we can define the restriction of $\mathcal L_\mathcal Z$ to $N$ just by taking
$\mathcal L_\mathcal Z\vert_N=\{(Z\vert_N,d)\}$, where $Z\vert_N$ is the restriction of the principal ideal sheaf $Z$ to $N$. In this way we obtain an immersed idealistic space $\mathcal Z\vert_N$ given by
$$
\mathcal Z\vert_N=(M,E,N,\mathcal L_\mathcal Z\vert_N).
$$

\begin{definition}
\label{def: monomial flower}
Let $\mathcal P=\{\mathcal V_\alpha\}_{\alpha\in \Lambda}$ be an immersed idealistic $e$-atlas over $(M,E)$, where $\mathcal V_\alpha=(M_\alpha,E_\alpha,N_\alpha,\mathcal L_\alpha)$. We say that $\mathcal P$ is {\em monomial}, or {\em quasi-ordinary}, if there is a logarithmic idealistic space $\mathcal Z$ over $(M,E)$ such that, for any $\alpha\in \Lambda$, the immersed idealistic $e$-spaces $\mathcal V_\alpha$ and $\mathcal Z\vert_{N_\alpha}$ are equivalent. An idealistic $e$-flower $\mathcal F$ is {\em monomial} or {\em quasi-ordinary} if it contains a monomial immersed idealistic $e$-atlas.
\end{definition}

\begin{remark}
\label{rk: pzeta}
If $\mathcal P$ is monomial, then the family
$
\mathcal P_\mathcal Z=
\{\mathcal Z\vert_{N_\alpha}\}_{\alpha\in \Lambda}$
is an immersed idealistic $e$-atlas equivalent to $\mathcal P$, and conversely.
\end{remark}

The reduction of singularities of logarithmic idealistic spaces is of combinatorial nature and well-known. It is a direct consequence of Hironaka's Game, solved by Spivakovsky \cite{Spi}. The precise statement we need is the following one:
\begin{proposition}
\label{prop:redsinglogspaces}
 Let $\mathcal Z=(M,E,\{(Z,d)\})$ be a logarithmic idealistic space over $(M,E)$. There is a reduction of singularities for $\mathcal Z$ obtained by the composition of a finite sequence of permissible blowing-ups, where the centers are connected components of intersections of the irreducible components of the divisor $E$ in the support of $Z$.
\end{proposition}
\begin{proof}
See, for example \cite[Proposition 3]{Mol}.
\end{proof}
\begin{corollary}
There is a reduction of singularities for any monomial idealistic $e$-flower $\mathcal F$.
\end{corollary}
\begin{proof}
Let $\mathcal Z$ be a logarithmic idealistic space associated to $\mathcal F$.
The blowing-ups of the reduction of singularities of $\mathcal Z$ induce a reduction of singularities for $\mathcal F$ as follows. Assume that $Y$ is a permissible center for $\mathcal Z$ and let $\pi:(M',E')\rightarrow (M,E)$ be the blowing-up with center $Y$. Up to taking an open subset of $M$, we can assume that there is a particular immersed idealistic $e$-space
$$
\mathcal Z\vert_N=(M,E,N,\mathcal L_\mathcal Z\vert_N)
$$
belonging to $\mathcal F$.  Since $N$ has normal crossings with $E$ and $Y$ is intersection of components of $E$ not containing $N$, the restriction of $\pi$ to the strict transform $N'$ of $N$ is the blowing-up
$$
\bar\pi:(N',E'\vert_{N'})\rightarrow (N,E\vert_N)
$$
with center $Y\cap N$, that is permissible for $(M,E,N,\mathcal L_\mathcal Z\vert_N)$. The necessary commutativity properties hold and we obtain a reduction of singularities for $\mathcal F$.
\end{proof}

\section{Adjusted Idealistic $e$-Flowers}

Let us consider an ambient space $(M,E)$ and an idealistic $e$-flower $\mathcal F$ over $(M,E)$. We recall that $\mathcal F$ is adjusted if $\delta_P\mathcal F=1$, for any $P\in \operatorname{Sing}\mathcal F$.

We have the following stability result:
\begin{proposition}
\label{prop:ingl estabilidadajustado}
Let $\mathcal F$ be an adjusted idealistic $e$-flower over $(M,E)$. Consider a  test system $\mathcal S$ that is permissible for $\mathcal F$ and let $\mathcal F'$ the transform of $\mathcal F$ by $\mathcal S$. Then, the idealistic $e$-flower $\mathcal F'$ is adjusted.
\end{proposition}
\begin{proof}
The statement for open projections is straightforward. In the case of permissible blowing-ups, the result is a direct consequence of the stability of the multiplicity of an hypersurface under blowing-up with equimultiple centers.
\end{proof}

\subsection{Logarithmic Factors}\label{Logarithmic factors}
Let $(M,E)$ be an ambient space and consider an immersed idealistic $e$-atlas $\mathcal P=
\{
\mathcal V_\alpha
\}_{\alpha\in \Lambda}
$
over $(M,E)$, where
\begin{equation}\label{eq:normalized atlas}
\mathcal V_\alpha=(M_\alpha,E_\alpha,N_\alpha,\mathcal L_\alpha), \quad
\mathcal L_\alpha=\{(I_{\alpha,j},d_{\alpha j})\}_{j=1}^{k_\alpha}.
\end{equation}
We say that $\mathcal P$ is {\em $d$-normalized} if   $d=d_{\alpha j}$, for any $\alpha\in \Lambda$, $j=1,2,\ldots,k_\alpha$.

\begin{remark} Any immersed idealistic $e$-atlas over $(M,E)$ is equivalent to a normalized one, see Subsection
\ref{ingl Ejemplos de espacios idealisticos equivalentes}.
\end{remark}

Assume now that $\mathcal P$ is $d$-normalized. A  {\em logarithmic factor $Z$ for $\mathcal P$} is a logarithmic sheaf $Z$ for $(M,E)$
such that there is a factorization
$$
I_{\alpha,j}=Z\vert_{N_\alpha}\cdot J_{\alpha,j}
$$
of principal ideal sheaves on $N_\alpha$, for any $\alpha\in \Lambda$ and $j=1,2,\ldots,k_\alpha$.

Let us note that there is at least one logarithmic factor given by the ideal sheaf $Z=\mathcal O_{M}$, which has empty support.

\begin{definition}
Let $\mathcal P$ be a $d$-normalized immersed idealistic $e$-atlas over $(M,E)$ and consider a logarithmic factor $Z$ for $\mathcal P$. The {\em co-factorial order $\mu_Z\mathcal P$} is given by
$$
\mu_Z\mathcal P=\max\{d\delta_P\mathcal P-\nu_PZ;\;  P\in \operatorname{Sing}\mathcal P\}.
$$
When $\operatorname{Sing}\mathcal P=\emptyset$, we put $\mu_Z\mathcal P=-\infty$.
\end{definition}

For any $P\in N_\alpha$, we have that $\nu_P(Z\vert_{N_\alpha})=\nu_PZ$, since $N_\alpha$ has normal crossings with $E$ and it is not locally contained in the support of $Z$. In particular, we have that
\begin{equation}
\label{eq:muuve}
d\delta_P\mathcal P-\nu_PZ=d\delta_P\mathcal V_\alpha-\nu_P(Z\vert_{N_\alpha}),
\end{equation}
when $P\in \operatorname{Sing}\mathcal P$. We conclude that $\mu_Z\mathcal P\in \mathbb Z_{\geq 0}\cup\{-\infty\}$. The condition $\mu_{Z}\mathcal P=-\infty$ means exactly that $\operatorname{Sing}\mathcal P=\emptyset$.

\begin{proposition}
\label{prop: ingl cofactorcero} Let $\mathcal P$ be a $d$-normalized immersed idealistic $e$-atlas over $(M,E)$. Assume that $Z$ is a logarithmic factor for $\mathcal P$ such that
$\mu_Z\mathcal P=0$.  Then $\mathcal P$ is monomial.
\end{proposition}
\begin{proof} Denote $\mathcal Z=(M,E,\{(Z,d)\})$ and for any $\alpha\in \Lambda$, denote
		$$
	\mathcal Z\vert_{N_\alpha}=
	(M_\alpha,E_\alpha,N_\alpha,\{(Z\vert_{N_\alpha},d)\}).
	$$
By Remark \ref{rk: pzeta},
it is enough to prove that the family
$
\mathcal P_\mathcal Z=
\{\mathcal Z\vert_{N_\alpha}\}_{\alpha\in \Lambda}
$
is an immersed idealistic $e$-atlas that is equivalent to $\mathcal P$.
	
Consider an immersed idealistic $e$-chart
	$$\mathcal V_\alpha=(M_\alpha,E_\alpha,N_\alpha,\mathcal L_{\alpha}=\{(I_{\alpha,j},d)\}_{j=1}^{k_\alpha})
	$$
	belonging to  $\mathcal P$. Let us prove that $\mathcal V_\alpha$ is equivalent to $\mathcal Z\vert_{N_\alpha}$.
	By the local  character of the equivalence, it is enough to do it locally at any point $P\in N_\alpha$ (see Proposition
	\ref{prop:ingl caracterlocalequivalencia}).
	
	Assume  that $P\notin \operatorname{Sing}\mathcal V_\alpha$.
There is a $j_0$ such that
	$$
	d>\nu_P(I_{\alpha,j_0})=\nu_PZ+\nu_PJ_{\alpha,j_0}\geq \nu_PZ.
	$$
 Then $P\notin \operatorname{Sing}\mathcal Z\vert_{N_\alpha}$ and  we are done.

 Assume that $P\in \operatorname{Sing}\mathcal V_\alpha$. Since $\mu_Z\mathcal P=0$,
 we have the equality of germs
	$
	(I_{\alpha,j_0})_P=(Z\vert_{N_\alpha})_P
	$, for an index $j_0$.
	This implies that, locally at $P$, the immersed idealistic spaces $\mathcal V_\alpha$ and
	$
	\mathcal Z\vert_{N_\alpha}
	$
	are equivalent, see the examples with redundant ideals in Subsection \ref{ingl Ejemplos de espacios idealisticos equivalentes}.
\end{proof}

\begin{corollary} If $\mu_Z\mathcal P=0$, then $\mathcal P$ has a reduction of singularities.
\end{corollary}

\subsection{Adjustment by Logarithmic Factors}
Consider a $d$-normalized immersed idealistic $e$-atlas $\mathcal P=
\{
\mathcal V_\alpha
\}_{\alpha\in \Lambda}
$
over $(M,E)$ as in Equation \eqref{eq:normalized atlas} and take a logarithmic factor   $Z$ for $\mathcal P$.
Take an integer number $m\geq 1$ with $m\geq\mu_Z\mathcal P$ and denote by
$\mathcal P^{Z,m}=\{ \mathcal V_\alpha^{Z,m}\}_{\alpha\in \Lambda}$ the family of immersed idealistic $e$-charts
$$
\mathcal V_\alpha^{Z,m}=(M_\alpha,E_\alpha,N_\alpha,\mathcal L_\alpha^{Z,m}), \quad \mathcal L_\alpha^{Z,m}=\mathcal L_\alpha\cup \{(J_{\alpha,j},m)\}_{j=1}^{k_\alpha}.
$$
Recall that $I_{\alpha,j}=Z\vert_{N_\alpha}J_{\alpha,j}$.
\begin{remark}
\label{rk:muzetalugarsingular}	
We have that
$\operatorname{Sing}(\mathcal V_\alpha^{Z,m})\subset \operatorname{Sing}(\mathcal V_\alpha)$.
The property
$$m>\mu_{Z}\mathcal P$$ holds if and only if  $\operatorname{Sing}(\mathcal V_\alpha^{Z,m})=\emptyset$, for any $\alpha\in \Lambda$. In particular, if  $m=\mu_Z\mathcal P$, there is an index $\alpha\in \Lambda$ such that  $\operatorname{Sing}(\mathcal V_\alpha^{Z,m})\ne\emptyset$.
\end{remark}

 The objective of this subsection is to prove the following result:
\begin{proposition}
\label{prop:ingl elajusteesajustado}
The family
$\mathcal P^{Z,m}=\{\mathcal V_{\alpha}^{ Z,m}\}_{\alpha\in \Lambda}$ is an adjusted immersed idealistic $e$-atlas over $(M,E)$.
\end{proposition}
The fact that $\mathcal P^{Z,m}$ is adjusted comes from the definition of $\mu_Z\mathcal P$. It remains to show that $\mathcal P^{Z,m}$ is an immersed idealistic $e$-atlas over $(M,E)$.

Recall the notation $M_{\alpha\beta}=M_\alpha\cap M_\beta$ and  denote $S_\alpha^{Z,m}=\operatorname{Sing}(\mathcal V_\alpha^{Z,m})$, for any pair of indices $\alpha,\beta\in\Lambda$.

\begin{lemma}
	\label{lema: singularlocus factor}
	$
	M_{\alpha\beta}\cap S_\alpha^{Z,m}=
	M_{\alpha\beta}\cap S_\beta^{Z,m}
	$.
\end{lemma}
\begin{proof} Consider a point $P\in 	M_{\alpha\beta}\cap S_\alpha^{Z,m}$. We have to show that
$$
\delta_P(\mathcal V_\beta)\geq 1,\quad \nu_P(J_{\beta,j})\geq m,\, j=1,2, \ldots,k_\beta.	
$$
Since $\delta_P(\mathcal V_\alpha)\geq 1$ and $\delta_P(\mathcal V_\beta)=\delta_P(\mathcal V_\alpha)$, then
$\delta_P(\mathcal V_\beta)\geq 1$.

On the other hand, we have that
 $$
d\delta_P\mathcal V_\alpha-\nu_PZ= m.
$$
Indeed, if $
d\delta_P\mathcal V_\alpha-\nu_PZ< m
$ we get $P\notin S_\alpha^{Z,m}$, which is not possible.
Now,  for any $j=1,2,\ldots,k_\beta$, we have
$$
\nu_P(J_{\beta,j})= \nu_P(I_{\beta,j})-\nu_PZ\geq
d\delta_P\mathcal V_\beta-\nu_PZ=
d\delta_P\mathcal V_\alpha-\nu_PZ= m.
$$
We conclude that $P\in S_\beta^{Z,m}$, as desired.
\end{proof}
In order to prove Proposition \ref{prop:ingl elajusteesajustado}, up to restrict ourselves to $M_{\alpha\beta}$,  it is enough to consider the particular case when
$$
\Lambda=\{\alpha,\beta\},\quad (M,E)=(M_\alpha,E_\alpha)=(M_\beta,E_\beta).
$$
Let us prove that $\mathcal V_\alpha^{Z,m}$ and $\mathcal V_\beta^{Z,m}$ are equivalent.
In view of Lemma 	\ref{lema: singularlocus factor}, we have that
$$
S^{Z,m}=
 \operatorname{Sing}(\mathcal V_\alpha^{Z,m})= \operatorname{Sing}(\mathcal V_\beta^{Z,m}).
$$
Now, it is enough to show that this situation repeats under open projections and permissible blowing-ups.  More precisely, let
$$
\sigma: (M',E')\rightarrow (M,E)
$$
be either an open projection or a blowing-up with a center $Y\subset S^{Z,m}$ that is a  non-singular closed analytic subspace of $M$  having normal crossings with $E$. We have to find a logarithmic factor $Z'$ for the transform  $\mathcal P'$ such that $\mu_{Z'}\mathcal P'\leq m$ and the following commutativity property holds:
\begin{equation}
	\label{eq:commuttivityblowingup}
(\mathcal V_\alpha^{Z,m})'=
(\mathcal V_\alpha')^{Z',m},\quad
(\mathcal V_\beta^{Z,m})'= (\mathcal V_\beta')^{Z',m}.
\end{equation}
Since $\mathcal V_\alpha'$ and $\mathcal V'_\beta$ are equivalent, the situation repeats, as desired.

The construction of $Z'$ in the case of an open projection is straightforward. In the case of a blowing-up, we take
$$
Z'=\mathcal J_{D'}^{m-d}\pi^{-1}Z, \quad D'=\pi^{-1}(Y).
$$
 This ends the proof of Proposition \ref{prop:ingl elajusteesajustado}.
\begin{remark}
\label{rk:commutativityexplosiones} The commutativity expressed in Equation  \eqref{eq:commuttivityblowingup} extends to $d$-normalized immersed idealistic $e$-atlas, in the sense that we have
\begin{equation}
	\label{eq:commuttivityblowingupatlas}
	(\mathcal P^{Z,m})'= (\mathcal P')^{Z',m},
\end{equation}
for the transforms under an open projection or a blowing-up with center that is permissible for $\mathcal P^{Z,m}$.
\end{remark}
\subsection{Reduction to the Adjusted Case} In this subsection, we give a proof of the statement corresponding to Equation
\eqref{eq:ingl reduccion al caso ajustado}.  We state the result as follows:
\begin{proposition}
	Let $\mathcal F$ be an idealistic $e$-flower over the ambient space $(M,E)$. Assume that the following statements are true:
	\begin{itemize}
		\item[a)] The monomial idealistic $e$-flowers have reduction of singularities.
		\item[b)] The adjusted idealistic $e$-flowers have reduction of singularities.
	\end{itemize}
Then $\mathcal F$ has reduction of singularities.
	\end{proposition}
\begin{proof}
Take a $d$-normalized immersed idealistic $e$-atlas $\mathcal P$ belonging to $\mathcal F$. Let us see that $\mathcal P$ has reduction of singularities.
Fix a logarithmic factor $Z$ for $\mathcal P$. We do induction on the co-factorial order $\mu_Z\mathcal P$. If $\mu_Z\mathcal P=0$, we are done, since $\mathcal P$ is monomial by  Proposition \ref{prop: ingl cofactorcero}.

Assume that $\mu_Z\mathcal P=m\geq 1$. We know that $\mathcal P^{Z,m}$ is an adjusted immersed idealistic $e$-atlas over $(M,E)$ such that
$$
\emptyset\ne \operatorname{Sing}(\mathcal P^{Z,m})\subset \operatorname{Sing}(\mathcal P).
$$
The permissible centers for $\mathcal P^{Z,m}$ are also permissible for $\mathcal P$. In view of our hypothesis, there is a sequence $\{\pi_i\}_{i=1}^p$ of permissible blowing-ups for $\mathcal P^{Z,m}$,
 such that
$$
\operatorname{Sing}\left((\mathcal P^{Z,m})^{(p)}\right)=\emptyset.
$$
By the commutativity in Remark \ref{rk:commutativityexplosiones}, the centers of $\pi_i$ are permissible for the successive transforms of $\mathcal P$ and we have that
$$
(\mathcal P^{Z,m})^{(p)}=
(\mathcal P^{(p)})^{Z^{(p)},m}.
$$
Hence  $\operatorname{Sing}(\mathcal P^{(p)})^{Z^{(p)},m}=\emptyset$ and thus $\mu_{Z^{(p)}}\mathcal P^{(p)}<m$. We end by the induction hypothesis.
\end{proof}

\section{Reduction to the Adjusted-Reduced Case}
Here we provide a proof of the reduction to the adjusted-reduced case, corresponding to Equation \eqref{eq:ingl reduccion al caso reducido}. We state the result as follows:
\begin{proposition}
	\label{prop:reduction a adjustedreduced}
	Let $\mathcal F$ be an adjusted idealistic $e$-flower over the ambient space $(M,E)$. Assume that the following statements are true:
	\begin{itemize}
		\item[a)] The idealistic $(e-1)$-flowers have reduction of singularities.
		\item[b)] The adjusted-reduced idealistic $e$-flowers have reduction of singularities.
	\end{itemize}
	Then $\mathcal F$ has reduction of singularities.
\end{proposition}

Let us recall that an idealistic $e$-flower $\mathcal F$ is called to be {\em reduced} if $\dim\operatorname{Sing}\mathcal F\leq e-2$. We have the following stability result:
	\begin{proposition}
		Let $\mathcal F$ be an idealistic $e$-flower over the ambient space $(M,E)$. Consider a permissible test system $\mathcal S$ for $\mathcal F$ and denote by $\mathcal F'$ the transform of $\mathcal F$ by $\mathcal S$. If $\mathcal F$ is adjusted and reduced, then $\mathcal F'$ is also adjusted and reduced.
	\end{proposition}
\begin{proof} The result is consequence of the fact that when we blow-up a center that is equimultiple for a hypersurface, then the exceptional divisor is not in the locus of maximal multiplicity. We leave the details to the reader.
\end{proof}

Let us start the proof of Proposition \ref{prop:reduction a adjustedreduced}. Take an adjusted idealistic $e$-flower $\mathcal F$ over $(M,E)$. If $\dim\operatorname{Sing}\mathcal F\leq e-2$, we are done. On the other hand we know that $\dim\operatorname{Sing}\mathcal F\leq e-1$, hence we can assume that  $\dim\operatorname{Sing}\mathcal F= e-1$. Let $L_1,L_2,\ldots,L_s$ be the irreducible components of dimension $\dim L_i=e-1$ of  $\operatorname{Sing}\mathcal F$. We reason by induction on the number $s$. Consider the following result:
\begin{proposition}
	\label{prop: aislar L}
Assume that $s\geq 1$. There is an open set $U\subset M$ such that
$
U\cap \operatorname{Sing}\mathcal F=L_1
$
and $L_1$ is a non-singular closed analytic subset $L_1\subset M$.
\end{proposition}
\begin{proof} Take a point $P\in L_1$. It is enough to show that $\operatorname{Sing}\mathcal F$ is non-singular at $P$. Take an immersed idealistic $e$-chart
	$$
	\mathcal V=(U, E\cap U, N,\mathcal L),\quad   \mathcal L=\{(I_j,d_j)\}_{j=1}^k,
	$$
	belonging to $\mathcal F$, with $P\in U$. We have that $\nu_{L_1\cap U}(I_j)\geq d_j$, for any $j=1,2,\ldots,k$. Since $\mathcal F$ is adjusted, there is an index $j_0$ such that
	$$
	\nu_PI_{j_0}=\nu_{L_1\cap U}I_{j_0}=d_{j_0}.
	$$
	This means that $L_1\cap U$ is $d_{j_0}$-equimultiple for $I_{j_0}$. Noting that $L_1\cap U$ is a hypersurface in $N$, we have that
	$$
	I_{j_0}=(\mathcal J_{L_1\cap U})^{d_{j_0}},
	$$
	(locally at $P$, that is, up to make $U$ smaller) where $\mathcal J_{L_1\cap U}$ is the ideal sheaf of $\mathcal O_N$ defining $L_1\cap U$. We conclude that the singular locus $\operatorname{Sing}\mathcal F$  is equal to $L_1\cap U$, near $P$, and it is non-singular.
\end{proof}
As a consequence of Proposition \ref{prop: aislar L} and in view of the induction hypothesis on $s$, we can restrict ourselves to the case when
$$
\operatorname{Sing}\mathcal F=L,
$$
where $L$ is a non-singular irreducible $(e-1)$-dimensional closed analytic subset of $M$. If $L$ has normal crossings with $E$, a single blowing-up with center $L$ makes $L$ to disappear from the singular locus and we are done. Now, it suffices to obtain the property that $L$ has normal crossings with $E$ by means of blowing-ups with centers contained in $L$. This is solved in next subsection.
\subsection{Normal Crossings for a Non-Singular Closed Subspace} The result we present here is what we need for the end of the proof of Proposition \ref{prop:reduction a adjustedreduced}:
\begin{proposition}
\label{prop: Normal crossings for a non-singular closed subspace}
	Consider an $n$-dimensional ambient space $(M,E)$ and a closed non-singular subset $L\subset M$ of dimension $e-1<n$. Let us assume that any idealistic $(e-1)$-flower has reduction of singularities.
	Then there is a finite sequence of blowing-ups
	\begin{equation}
		\label{eq:sequence}
	(M,E)\leftarrow (M_1,E_1)\leftarrow \cdots \leftarrow (M_k,E_k)
	\end{equation}
	of ambient spaces with centers contained in the successive strict transforms of $L$, in such a way that the last strict transform $L_k$ has normal crossings with $E_k$.
\end{proposition}
\begin{proof} Let us consider  $L$ as the support of a new ambient space. Write $E=E^*\cup D\cup F$, where  $F$ is the union of components of $E$ containing $L$ and the divisor $D$ is union of other components of $E$, in such a way that $F\cup D$ has normal crossings with $L$.
Let us write $E^*=\cup_{i\in A}E^*_i$ the decomposition of $E^*$ into its irreducible components. For any $B\subset A$, let us denote $E^*_B=\bigcap_{i\in B}E^*_i$.
Consider the set
$$
\Sigma=\{B\subset A;\; E^*_B\cap L\ne\emptyset\}.
$$
Let $s=\max\{\#B;\; B\in \Sigma\}$,  $\Sigma_s=\{B\in \Sigma;\; \#B=s\}$ and $t=\#\Sigma_s$. If $s=0$, we are done, since then $L\cap E^*=\emptyset$. We reason by induction on the lexicographical invariant $(s,t)$. Take $B\in \Sigma_s$ and consider the list of marked ideals in $L$ given by
$$
\mathcal L_{M,E^*,L,B}=\{(I_{E^*_i}\vert_L, 1)\}_{i\in B}.
$$
Take the idealistic space $\mathcal N$ over $(L,L\cap D)$ defined as
$$
\mathcal N=
\mathcal N_{M,E,E^*,L,B}=(L,L\cap D,\mathcal L_{M,E^*,L,B}).
$$
Note that $\operatorname{Sing}\mathcal N=E^*_B\cap L$.

Let $Y$ be a permissible center for $\mathcal N$. We have that:
	\begin{enumerate}
		\item $Y\subset L$ is non-singular.
		\item $Y\subset F_i$, for each irreducible component of  $F$.
		\item $Y$ has normal crossings with  $L\cap D$.
		\item $Y$ is equimultiple for $I_{E^*}$, where $I_{E^*}$ is the ideal sheaf of $\mathcal O_M$ defining $E^*$. More precisely, we have that $Y\subset E^*_B$ and for any index $i\in A\setminus B$, we have that $E^*_i\cap Y=\emptyset$.
	\end{enumerate}
We conclude that $Y$ is equimultiple for
	$I_{F\cup E^*}=I_{E^*}I_F$ and it has normal crossings with  $D$. Then, we have that $Y$ has normal crossings with $E$. Thus, we can perform the blowing-up of ambient spaces
$$
\pi_Y:(M',E')\rightarrow (M,E)
$$
centered at $Y$, with $E'=(E^*)'\cup \tilde D \cup F'$, where $\tilde D=\pi_Y^{-1}(D\cup Y)$ and $(E^*)'$, $F'$ are the respective strict transforms of $E^*$, $F$. If  $L'$ is the strict transform of $L$, the restriction
$$
\bar\pi_Y: (L',L'\cap \tilde D)\rightarrow (L,L\cap D),
$$
is the blowing-up of $(L,L\cap D)$ with center $Y$.

Let $\mathcal N'$ be the transform of $\mathcal N$ by $\bar\pi_Y$. If $\operatorname{Sing}(\mathcal N')=\emptyset$, the new inva\-riant $(s',t')$ is strictly smaller than $(s,t)$ and we are done by induction.
If  $\operatorname{Sing}(\mathcal N')\ne\emptyset$, then
 $(s',t')=(s,t)$, we have that
$A'=A$, $\Sigma'=\Sigma$
and the following  commutativity property holds:
$$
\mathcal N'=\mathcal N_{M',E',(E^*)',L',B}.
$$
Now, we end by performing a reduction of singularities of $\mathcal N$.
\end{proof}

\strut\vfill\pagebreak
\part{Projections of Idealistic Exponents}
\label{Projections of Idealistic Exponents}
Let $\mathcal E$ be an adjusted and reduced idealistic exponent over an $n$-dimensional ambient space $(M,E)$ and consider a  hypersurface $(M,E,H)$. In this part, we introduce a procedure for obtaining a new idealistic exponent $\operatorname{pr}_{H}\mathcal E$ over $(H, E\vert_H)$ that we call {\em the projection of $\mathcal E$ over $(M,E,H)$} whose main properties are the following ones:
\begin{itemize}
	\item $
	H\cap \operatorname{Sing}\mathcal E= \operatorname{Sing}(\operatorname{pr}_{H}\mathcal E)
	$.
	\item A closed analytic subset $Y\subset H$ is a permissible center for $\operatorname{pr}_{H}\mathcal E$ if and only if it is a permissible center for $\mathcal E$.
	\item $
	(\operatorname{pr}_H\mathcal E)'= \operatorname{pr}_{H'}(\mathcal E')
	$, for the transforms under a morphism
	that is either an open projection or a blowing-up with permissible center $Y$ for $\mathcal E$, with $Y\subset H$.
\end{itemize}
Let us note that the assignment $\left((M,E,H), \mathcal E\right)\mapsto \operatorname{pr}_H\mathcal E$ is necessarily unique if it exists. Indeed, we deduce from the above properties that if we have two assignments $\operatorname{pr}$ and $\widetilde{\operatorname{pr}}$, then $\operatorname{pr}_H\mathcal E=\widetilde{\operatorname{pr}}_H\mathcal E$, since they have the same permissible test systems.

Let us also note that $(M,E,H)$ may be transverse or not. Thus, there is a union $E^*$ of irreducible components of $E$ such that $E\vert_H=E^*\cap H$.

We end this part with the construction of the projection of an idea\-listic  $e$-flower over $(M,E,H)$ in the particular case that $H$ is a disjoint union of irreducible components of $E$, with the same properties as above.

Once  the projections are established, the existence of a reduction of singularities of $\operatorname{pr}_H\mathcal E$ implies the existence of a morphism $$\sigma:(M',E')\rightarrow (M,E)$$
that is a composition of a finite sequence of permissible blowing-ups for $\mathcal E$ in such a way that $H'\cap \operatorname{Sing}\mathcal E'=\emptyset$, where $H'$ is the strict transform of $H$ by $\sigma$. This is a key observation for the Maximal Contact Theory.

The projections are done with the help of {\em projecting axes}. We have built these structures inspired in a part of the work of Panazzolo in \cite{Pan}.

\section{Projecting Axes}
The construction of projecting axes is done ``around $H$'' instead of considering ``the whole ambient space''. More precisely, we will consider the open set $(M_H,E_H)$ of $(M,E)$ defined as follows. Recalling that $M$ is a germ over the compact set $K\subset M$, we define $M_H$ to be the germ of $M$ over the compact set $K\cap H$ and $E_H$ to be the germ of $E$ over the compact set $E\cap K\cap H$.

 Denote by
$\Theta_{M}[\log E^*]$ the sheaf
 of germs of vector fields over $M$ that are tangent to $E^*$.
A {\em projecting chart for  $(M,E,H)$} is a pair  $$
\mathfrak c=(U,\xi),
$$ where $U$ is an open set of $M_H$ and $\xi\in \Theta_{M}[\log E^*](U)$ is a non-singular vector field transverse to $H\cap U$ at every point. Two projecting charts $\mathfrak c_1=(U_1,\xi_1)$ and $\mathfrak c_2=(U_2,\xi_2)$  for $(M,E,H)$ are {\em compatible} if there is a unit $u_{12}$, defined in $U_{12}=U_1\cap U_2$, such that
$$
\xi_2\vert_{U_{12}}=u_{12}\xi_1\vert_{U_{12}},\quad \xi_1\vert_{U_{12}}(u_{12})=0.
$$
Note that we  ask $u_{12}$ to be a first integral of $\xi_1\vert_{U_{12}}$. Automatically, we have that $u_{12}$ is also a first integral of  $\xi_2\vert_{U_{12}}$.

Given $P\in H$, there is always a projecting chart  $\mathfrak c=(U,\xi)$, with $P\in U$. It is enough to take local coordinates $\mathbf x,z$ around $P$ defined on  $U$ and adapted to $E$, such that $H\cap U=(z=0)$ and
$$\xi=\partial/\partial z.$$
Conversely, given a projecting chart  $\mathfrak c=(U,\xi)$ and a point $P\in H\cap U$, there are local coordinates $\mathbf x,z$ around $P$ defined on an open set $V\subset U$ that are adapted to $E$, such that $H\cap U=(z=0)$ and $\xi=\partial/\partial z$.

This suggests the following definition:
\begin{definition} A {\em rectified projecting chart  for $(M,E,H)$} is a data $(\mathfrak c, \mathbf x,z)$, where $\mathfrak c=(U,\xi)$ is a projecting chart  for $(M,E,H)$ and $\mathbf x,z$ are coordinate functions defined on $U$, adapted to $E$, such that $H\cap U=(z=0)$ and that $\xi=\partial/\partial z$.  A given projecting chart $\mathfrak c$  for $(M,E,H)$ is {\em recti\-fiable} if there are coordinates $\mathbf x,z$ such that $(\mathfrak c,\mathbf x,z )$ is a rectified projecting chart  for $(M,E,H)$.
\end{definition}

We define a {\em projecting atlas $\mathfrak a$  for $(M,E,H)$} to be a finite family
$\mathfrak a=\{\mathfrak c_\alpha\}_{\alpha\in \Lambda}
$,
 such that the $\mathfrak c_\alpha$ are two by two compatible projecting charts  for $(M,E,H)$, whose definition domains cover $H$; equivalently, the definition domains cover $M_H$.
\begin{definition}
\label{def:projectingaxes}
	Two projecting atlases $\mathfrak a_1$ and $\mathfrak a_2$ of $(M,E)$ over the hypersurface  $(M,E,H)$ are {\em compatible} if their union $\mathfrak a_1\cup\mathfrak a_2$ is also a projecting atlas  for $(M,E,H)$.
	The compatibility classes of projecting atlases are called {\em projecting axes  for $(M,E,H)$}. Given a projecting axis $\mathfrak E$ and a projecting chart $\mathfrak c=(U,\xi)$  for $(M,E,H)$, we say that {\em $\mathfrak c$ belongs to $\mathfrak E$} if $\mathfrak c$ belongs to some of the projecting atlases defining $\mathfrak E$.
\end{definition}

Given a projecting atlas $\mathfrak a$, there is another compatible projecting atlas $\widetilde{\mathfrak a}$ such that all the charts in $\widetilde{\mathfrak a}$ are rectifiable. In particular, there is always an atlas composed of rectifiable charts among the atlases defining a given projecting axis.

\begin{remark}	
\label{obs:ingl localizaciondeejesenH}
A projecting axis  for $(M,E,H)$ is exactly the same object as a
projecting axis for $(M_H,E_H,H)$.
\end{remark}
\subsection{First Integrals of Projecting Axes}
Let $\mathfrak E$ be a projecting axis  for $(M,E,H)$. We denote by $\mathcal O_{M_H}$ the sheaf of germs of holomorphic functions of $M_H$, that is $\mathcal O_{M_H}=\mathcal O_{M}\vert_{M_H}$. Let us build the
{\em sheaf of first integrals $\operatorname{Int}\negthinspace\mathfrak E$ of $\mathfrak E$}. For each open subset $V\subset M_H$, we define $\operatorname{Int}\negthinspace\mathfrak E(V)$ to be the subring
$$
\operatorname{Int}\negthinspace\mathfrak E(V)\subset \mathcal O_{M_H}(V)=\mathcal O_M(V)
$$
whose elements are the holomorphic functions $h$ defined in $V$ satisfying the following equivalent properties:
\begin{enumerate}
\item[a)] Given a point $P\in V$, there is a projecting chart $\mathfrak c=(U,\xi)$ belonging to  $\mathfrak E$ such that $P\in U\subset V$ and $\xi(h\vert_U)=0$.
\item[b)] Given a projecting chart $\mathfrak c=(U,\xi)$ belonging to  $\mathfrak E$ such that $U\subset V$, we have that $\xi(h\vert_U)=0$.
\end{enumerate}
The sheaf $\operatorname{Int}\negthinspace\mathfrak E$ is a subsheaf of rings of $\mathcal O_{M_H}$. In particular, we have that $\mathcal O_{M_H}$ is a $\operatorname{Int}\negthinspace\mathfrak E$-module.

\begin{remark} Let $(\mathfrak c, \mathbf x,z)$ be a rectified
 projecting chart  for $(M,E,H)$ such that $\mathfrak c=(U,\xi)$ belongs to a given projecting axis $\mathfrak E$  for $(M,E,H)$. A function $h\in \mathcal O_M(U)$ is a first integral for $\mathfrak E$ if and only if $\partial h/\partial z=0$.
\end{remark}

\subsection{Local Nature of Projecting Axes.}
Let  $\mathfrak c=(U,\xi)$ be a projec\-ting chart  for $(M,E,H)$. The
{\em restriction $\mathfrak c\vert_{V}$ of $\mathfrak c$ to an open set $V$ of $M_H$} is given in a natural way by $$
\mathfrak c\vert_{V}=(U\cap V,\xi\vert_{U\cap V}).
$$
It is a projecting chart for $(V,E\cap V,H\cap V)$.

Let $\mathfrak c_1$ and $\mathfrak c_2$ be two projecting charts  for $(M,E,H)$. The following properties are equivalent:
\begin{enumerate}
\item The charts $\mathfrak c_1$ and $\mathfrak c_2$ are compatible.
\item For any open set $V\subset M_H$, the restrictions $\mathfrak c_1\vert_V$ and $\mathfrak c_2\vert_V$ are compatible.
\item  The restrictions $\mathfrak c_1\vert_V$ and $\mathfrak c_2\vert_V$ are compatible for the open sets $V\subset M_H$ belonging to an open cover of $M_H$.
\end{enumerate}

Consider a projecting atlas $\mathfrak a=\{\mathfrak c_\alpha\}_{\alpha\in \Lambda}$   for $(M,E,H)$.
The {\em restriction $\mathfrak a\vert_{V}$ of $\mathfrak a$ to an open set $V$ of $M_H$}  is given by
$$
\mathfrak a\vert_{V}=\{\mathfrak c_\alpha\vert_{V}\}_{\alpha\in \Lambda}.
$$
It is a projecting atlas for $(V,E\cap V,H\cap V)$. Moreover, if $\mathfrak a_1$ and $\mathfrak a_2$ are two compatible projecting atlases  for $(M,E,H)$, their restrictions $\mathfrak a_1\vert_{V}$ and $\mathfrak a_2\vert_{V}$ are also compatible. This allows us to define without ambiguity the {\em restriction $\mathfrak E\vert_V$} to $V$ of a projecting axis $\mathfrak E$  for $(M,E,H)$. It is, of course, a projecting axis for $(V,E\cap V,H\cap V)$.

The following ``gluing'' result, concerning the local nature of pro\-jecting axes, follows directly from the local nature of the compatibility of charts:
\begin{lemma} \label{lema:ingl pegado de ejes}
Consider an open covering $\{U_\beta\}_{\beta\in B}$ of $M_H$. Assume that for each $\beta\in B$, there is a projecting axis $\mathfrak E_\beta$ for $(U_\beta, E\cap U_\beta, H\cap U_\beta)$. If the equality
	$$
	\mathfrak E_\beta\vert_{U_{\beta\gamma}}=  \mathfrak E_\gamma\vert_{U_{\beta\gamma}},\quad U_{\beta\gamma}=U_\beta\cap U_\gamma,
	$$
	holds for any $\beta,\gamma\in B$, there is a unique projecting axis $\mathfrak E$  for $(M,E,H)$ such that $\mathfrak E\vert_{U_\beta}=\mathfrak E_\beta$, for all $\beta\in B$.
\end{lemma}

\subsection{Projections over the First Factor and Projecting Axes}
Consider a projection over the first factor
$$
\sigma:(M',E')=(M\times(\mathbb C^m,0), E\times(\mathbb C^m,0))\rightarrow (M,E)
$$
and put $H'=\sigma^{-1}(H)$.  Note that $M'_{H'}=\sigma^{-1}(M_{H})$ and that $(M',E',H')$ is a hypersurface of $(M',E')$.  Recall that we have the functions
$$
\omega_i:M'\rightarrow (\mathbb C,0),\quad i=1,2,\ldots,m,
$$
obtained by composition of
$M\times(\mathbb C^m,0)\rightarrow (\mathbb C^m,0)$ with the $i$-th coordinate function of $(\mathbb C^m,0)$.

Let $\mathfrak c=(U,\xi)$ be a projecting chart  for $(M,E,H)$. The {\em transform $\sigma^{-1}\negthinspace\mathfrak c=(U', \xi')$ of $\mathfrak c$ by $\sigma$} is obtained  by putting $U'=\sigma^{-1}(U)$ and by taking $\xi'$ to be the unique vector field over $U'$ satisfying
\begin{enumerate}
\item $
(d\sigma)\circ \xi'=\xi\circ \sigma
$. (The vector fields $\xi$ and $\xi'$ are $\sigma$-related).
\item $\xi'(\omega_i)=0$, for all $ i=1,2,\ldots,m$.
\end{enumerate}
We have that $\sigma^{-1}\negthinspace\mathfrak c$ is a projecting chart for $(M',E',H')$. If $\mathfrak c_1$ and $\mathfrak c_2$  are two compatible projecting charts, then, the transformed charts $\sigma^{-1}\negthinspace\mathfrak c_1$ and $\sigma^{-1}\negthinspace\mathfrak c_2$
are also compatible. In this way, we define the {\em transform $\sigma^{-1}\mathfrak a$ of a projecting atlas $\mathfrak a$}  as well as the
{\em transform $\sigma^{-1}\negthinspace\mathfrak E$ of a projecting axis $\mathfrak E$  for $(M,E,H)$}. We obtain, respectively, a projecting atlas and a projecting axis for $(M',E',H')$.
\begin{remark}
Let $\mathfrak E'=\sigma^{-1}\negthinspace\mathfrak E$. The sheaf of first integrals $\operatorname{Int}\negthinspace\mathfrak E'$ is related with $\operatorname{Int}\negthinspace\mathfrak E$ as follows. Consider an open subset $V'\subset M'_{H'}$. Note that $V'$ is of the form $V'=U\times (\mathbb C^m,0)$. A function $h'$ defined over $V'$ is a first integral for $\mathfrak E'$ if and only if it factorizes through a first integral $h$ of $\mathfrak E$ defined in $U$, that is $h'=h\circ \sigma$.
\end{remark}
\subsection{Blowing-up Projecting Axes}
Before studying the transformation of a projecting axis by a blowing-up, let us consider the following lemma:
\begin{lemma}
	\label{lema:ingl determinacionejesporabiertodenso}
	Let $\mathfrak E_1$ and $\mathfrak E_2$ be two projecting axis  for $(M,E,H)$. Assume that there is a closed analytic subset $Z\subset M_H$ of codimension greater than or equal to one, such that $\mathfrak E_1\vert_{M_H\setminus Z}= \mathfrak E_2\vert_{M_H\setminus Z}$. Then $\mathfrak E_1=\mathfrak E_2$.
\end{lemma}
\begin{proof}
	It is enough to show that any two projecting charts $\mathfrak c_1=(U_1,\xi_1)$ and $\mathfrak c_2=(U_2,\xi_2)$ of $\mathfrak E_1$ and $\mathfrak E_2$, respectively, are compatible.
	
	Write $U=U_1\cap U_2$. If $U=\emptyset$, there is nothing to prove. Assume that $U\ne\emptyset$ and denote $\overline{\mathfrak c}_i={\mathfrak c}_i\vert_U=(U,\bar\xi_i)$, for $i=1,2$. We need to show the existence of a unit $u\in \mathcal O_{M}(U)$ satisfying that $\bar\xi_2=u\bar\xi_1$, with $\bar\xi_1(u)=0$. Consider the subset
	$$
	A=\{P\in U; \; \text{ there is  } f_P\in \mathcal O_{M,P} \;\text{ such that } \xi_{2,P}=f_P\xi_{1,P}\}.
	$$
	Let us see that $A$ is a closed analytic subset of $U$. Recalling that $\bar\xi_1$ and $\bar\xi_2$ are non-singular vector fields, the set $A\subset U$ is defined by the equation
	$
	\bar\xi_1\wedge\bar\xi_2=0
	$.
Hence $A\subset U$ is a closed analytic subset. Let us write $W=U\setminus Z$. Since $\mathfrak c_1$ and $\mathfrak c_2$ are compatible charts in $W$, {we have that $W\subset A\subset U$, and hence $A=U$.} Then, given $P\in U$, we have the relation
	$$
	\xi_{2,P}=f_P\xi_{1,P}.
	$$
	Once again, since they are non-singular vector fields, we conclude that the germs $f_P$ are unique and they are units. These germs are ``glued'' in a unit $u\in \mathcal O_M(U)$ such that $u_P=f_P$, for all $P\in U$. Hence $\bar\xi_2=u\bar\xi_1$.
	
	On the other hand, the compatibility of $\mathfrak c_1$ and $\mathfrak c_2$ in $W$ implies that
	$$
	\bar\xi_1(u)=0, \text{ over } W.
	$$
	Again, since $\bar\xi_1(u)=0$ defines a closed analytic subset of $U$, we have that $\bar\xi_1(u)=0$, in the whole $U$. We conclude that $\mathfrak c_1$ are $\mathfrak c_2$ compatible charts in $U$ as desired.
\end{proof}

Let $Y$ be a non-singular irreducible closed analytic subset of $H$ having normal crossings with $E$. Assume that the codimension of $Y$ in $H$ is grater than or equal to one and thus $Y$ does not coincide with any connected component of $H$. Let us perform the blowing-up
$$
\pi:(M',E')\rightarrow (M,E)
$$
centered at $Y$ and denote by $H'$ the strict transform of $H$ by $\pi$. Then $(M',E',H')$ is a hypersurface of $(M',E')$ and $\pi$ induces a blowing-up
$$
\bar\pi:(H',E'\vert_{H'})\rightarrow (H,E\vert_H).
$$
Let us note that $M'_{H'}\subset \pi^{-1}(M_H)$ and that $\pi$ induces an identification
between $M'_{H'}\setminus \pi^{-1}(Y)$ and $\pi(M'_{H'})\setminus Y$, as a consequence of the identification between $M'\setminus \pi^{-1}(Y)$ and $M\setminus Y$.

\begin{proposition}
\label{prop:transformofaxisbyblowingup}
	Let $\mathfrak E$ be a projecting axis  for $(M,E,H)$. There is a unique projecting axis $\mathfrak E'$ for $(M',E',H')$ such that
	$$
	\mathfrak E'\vert_{M'_{H'}\setminus \pi^{-1}(Y)}= \mathfrak E\vert_{\pi(M'_{H'})\setminus Y},
	$$
	where we have taken the identification $M'\setminus \pi^{-1}(Y)\rightarrow M\setminus Y$ induced by the blowing-up $\pi$.
\end{proposition}
\begin{proof}
	Uniqueness is a direct consequence of Lemma \ref{lema:ingl determinacionejesporabiertodenso}. Let us see the existence. In order to do it, we are going to prove the existence of a covering of $M_H$ by open subsets $U \subset M_H$ with the following property:
	\begin{quote}
		There is a projecting axis $\mathfrak E'$ for  $(U^\star,E^\star, H^\star)$
		such that
		\begin{equation} \label{eq:ingl pesadilla}
			\mathfrak E'\vert_{U^\star_{H^\star}\setminus \pi^{-1}(Y)}=
			\mathfrak E\vert_{\pi(U^\star_{H^\star})\setminus Y},
		\end{equation}
		where $U^\star=\pi^{-1}(U)$ and $E^\star$, $H^\star$ are the corres\-ponding intersections of $E'$, $H'$ with $U^\star$.
	\end{quote}
	Assume this result is proved and take two of these open subsets $U_\beta$ and $U_\gamma$. By invoking the uniqueness, we have the equality
$$
\mathfrak E'_\beta\vert_{U^\star_{\beta\gamma}}=\mathfrak E'_\gamma\vert_{U^\star_{\beta\gamma}},
\quad U^\star_{\beta\gamma}=U^\star_\beta\cap U^\star_\gamma.
$$
 Now, we obtain $\mathfrak E'$ by application of the gluing Lemma \ref{lema:ingl pegado de ejes}.
	
	 Around each point of $H$, we can choose a rectified projecting chart $(\mathfrak c=(U,\xi), \mathbf x,z)$ of $\mathfrak E$ with the additional condition that if $Y\cap U\ne\emptyset$, then
$$
Y\cap U=(z=0)\cap (x_i=0;\; i=1,2,\ldots,t),
$$
where $t\geq 1$ is the codimension of $Y$ in $H$.
In this way, we cover $M_H$ by the domains of that charts. It remains to prove that given one of that charts $(\mathfrak c=(U,\xi), \mathbf x,z)$, there is a projecting axis $\mathfrak E'$ for $(U^\star,E^\star, H^\star)$ satisfying to the property stated in Equation \eqref{eq:ingl pesadilla}.

	 Let us consider $U^\star_i =U^\star\setminus L^\star_i$, where $L^\star_i$ is the strict transform of $(x_i=0)$ by $\pi$, for any $i=1,2,\ldots,t$. Note that
	$$
	\textstyle H^\star \subset U^\star_{H^\star} \subset \widetilde{U}, \, \text{ where } \widetilde{U}=\bigcup_{i=1}^tU^\star_i.
	$$
	Let us write $V_i=U\setminus (x_i=0)$ and $W= \cup_{i=1}^tV_i$. We have the equality
	$$
	\pi(U^\star_i\setminus \pi^{-1}(Y))=V_i,\quad i=1,2,\ldots,t.
	$$
	Then, we have $\textstyle \pi(\widetilde{U}\setminus \pi^{-1}(Y))= W$. Since $ U^\star_{H^\star}\subset \widetilde{U}$, it is enough to prove the existence of a projecting axis  $\mathfrak E'$ defined in $\widetilde{U}$ satisfying that
	$$
	\mathfrak E'\vert_{\widetilde{U}\setminus \pi^{-1}(Y)}=\mathfrak E\vert_W,
	$$
	where we have taken the identification $\widetilde{U}\setminus \pi^{-1}(Y)\rightarrow W$ induced by the blowing-up $\pi$ outside of $\pi^{-1}(Y)$.

In order to obtain $\mathfrak E'$, we are going to define $\mathfrak E'_i$ on $U^\star_i$  by
\begin{equation}
\label{eq:pegadoejesblowingup}
\mathfrak E'_i\vert_{U^\star_i\setminus\pi^{-1}(Y)}=\mathfrak E\vert_{V_i}, \quad  i=1,2,\ldots,t.
\end{equation}
Indeed, in view of Lemma \ref{lema:ingl determinacionejesporabiertodenso}, the property in Equation \eqref{eq:pegadoejesblowingup} implies that
$$
\mathfrak E'_i\vert_{U^\star_i\cap U^\star_j}=\mathfrak E'_j\vert_{U^\star_i\cap U^\star_j},\quad i,j\in \{1,2,\ldots,t\}
$$
and  we obtain the desired $\mathfrak E'$ by gluing the $\mathfrak E'_i$, for $i=1,2,\ldots,t$.

Let us define the projecting axis $\mathfrak E'_i$. Consider the vector field $$\xi_i=x_i\xi,$$ defined in $U$, and the projecting chart $\mathfrak c_i=(V_i,\xi_i\vert_{V_i})$. Since $x_i$ is a first integral and a unit in $V_i$, the chart $\mathfrak c_i$ belongs to the axis $\mathfrak E$. Thus, the chart $\mathfrak c_i$ defines exactly $\mathfrak E\vert_{V_i}$. The vector field $\xi_i$ lifts by $\pi$ to a unique vector field $\xi_i'$ over $U_i^\star$, that is written in appropriate coordinates as
$$
\xi_i'=\partial/\partial z',
$$
where $H^\star\cap U^\star_i=(z'=0)$. This allows us to define $\mathfrak E'_i$ from the projecting chart $(U^\star_i,\xi'_i)$.
\end{proof}

 The definition of  {\em the transform $\mathfrak E'$ of $\mathfrak E$ by the blowing-up $\pi$} follows straightforward from
 Proposition \ref{prop:transformofaxisbyblowingup}.

\section{Projections of Idealistic Spaces}
Consider a hypersurface $(M,E,H)$ of an ambient space $(M,E)$. Let us take an adjusted and  reduced idealistic space $\mathcal M$ of $(M,E)$.

In this section, we construct projections of $\mathcal M$ over $(M,E,H)$ asso\-ciated to projecting axes and projectable generators of the marked ideals of $\mathcal M$. These constructions are compatible with the equivalence of idea\-listic spaces, with open projections and with permissible blowing-ups whose centers are contained in $H$. Thanks to these properties, in the next section,  we will be able to build the projection over $(M,E,H)$ for adjusted-reduced idealistic exponents of $(M,E)$.
\subsection{Projecting Systems. Local Nature}
\label{Projecting systems}
Before giving the definition of projecting systems, we need some concepts concerning only a
 projecting axis $\mathfrak E$ for $(M,E,H)$.

We say that a sheaf $\mathbf J\subset \mathcal O_{M_H}$ is an {\em $\mathfrak E$-projectable module over $M_H$} if it is a locally principal sub-module of $\mathcal O_{M_H}$ over $\operatorname{Int}\negthinspace{\mathfrak E}$. This means that there is an open covering $\{U_\beta\}_{\beta\in B}$ of $M_H$, jointly with holomorphic functions $F_\beta\in \mathcal O_M(U_\beta)$, satisfying the following properties:
\begin{enumerate}
	\item The germ of $F_\beta$ is non-zero at each point.
	\item For each open subset $V$ of $M_H$ and any index $\beta\in B$, we have
	$$
	\mathbf J(V\cap U_\beta)= \operatorname{Int}\negthinspace{\mathfrak E}(V\cap U_\beta)\cdot F_\beta\vert_{V\cap U_\beta}\subset \mathcal O_M(V\cap U_\beta).
	$$
\end{enumerate}
Note that, for each pair of indices $\beta,\gamma\in B$, we have that
$$
F_\gamma\vert_{U_\beta\cap U_\gamma}=u_{\beta\gamma}F_\gamma\vert_{U_\beta\cap U_\gamma},
$$
where $u_{\beta\gamma}\in \operatorname{Int}\negthinspace{\mathfrak E}(U_\beta\cap U_\gamma)$ is a first integral that is also a unit; in particular, it is also a unit in $\mathcal O_M(U_\beta\cap U_\gamma)$.

\begin{definition}
\label{def:projectablegenerator}
Let us consider a locally principal ideal sheaf $I\subset\mathcal O_{M_H}$. An {\em $\mathfrak E$-projectable generator $\mathbf J$ of $I$} is any  $\mathfrak E$-projectable module $\mathbf J$, with $\mathbf J\subset I$, that generates $I$ as $\mathcal O_{M_H}$-ideal sheaf.
\end{definition}

 In local terms, if $\mathbf J$ is given by a family $\{(U_\beta,F_\beta)\}_{\beta \in B}$ of open sets $U_\beta$ and functions $F_\beta\in \mathcal O_M(U_\beta)$ as before, the ideal $I(U_\beta)\subset \mathcal O_M(U_\beta)$  is generated by $F_\beta$, for each $\beta\in B$.

Now, we can define the projecting systems:
\begin{definition} \label{def:ingl sistemaproyectante}
	Let $\mathcal M=(M,E,\mathcal L)$ be an adjusted-reduced idealistic space. A {\em projecting system $\mathfrak S$ for $(\mathcal M,H)$} is a pair $\mathfrak S=(\mathfrak E,\mathfrak J)$ satisfying the following properties:
	\begin{enumerate}
		\item $\mathfrak E$ is a projecting axis for $(M,E,H)$.
		\item If $\mathcal L=\{(I_j,d_j)\}_{j=1}^k$, then $\mathfrak J$ is a list $\mathfrak J=\{\mathbf J_j\}_{j=1}^k$, where each $\mathbf J_j$ is an $\mathfrak E$-projectable generator of $I_j|_{M_H}$.
	\end{enumerate}
\end{definition}
Given a open set $V$ of $M_H$, the restriction $\mathfrak S\vert_V$ of a projecting system $\mathfrak S=(\mathfrak E,\mathfrak J)$ is naturally given by
$$
\mathfrak S\vert_V= (\mathfrak E\vert_V,\mathfrak J\vert_V),\quad \mathfrak J\vert_V=\{\mathbf J_j\vert_V\}_{j=1}^k.
$$
It is a projecting system for $(\mathcal M\vert_V, H\cap V)$.

We have local determination and  local gluing procedures for projecting systems, since the corresponding properties hold for $\mathcal M$, $\mathfrak E$ and $\mathfrak J$.

\subsection{Projections over the First Factor and Projecting Systems}
Consider a projection over the first factor
$$
\sigma:(M',E')=(M\times (\mathbb C^m,0), E\times (\mathbb C^m,0))\to (M,E)
$$
and an idealistic space
 $\mathcal M=(M,E,\mathcal L)$, with $\mathcal L=\{(I_j,d_j)\}_{j=1}^k$. Take  a {projecting system $\mathfrak S=(\mathfrak E,\mathfrak J)$  for $(\mathcal M,H)$.
The {\em transform $\mathfrak S'$ of $\mathfrak S$ by $\sigma$} will be denoted by
$$
\mathfrak S'=\left(\mathfrak E',\; \mathfrak J'=\{\mathbf J'_j\}_{j=1}^k\right),
$$
where  $\mathfrak E'$ is the transform of  $\mathfrak E$ by $\sigma$.

Let us detail which are the $\mathfrak E'$-projectable generators $\mathbf J'_j$ of $I'_j$.
Take an open subset  $U'$ of $M'_{H'}$. Since we are working with the germified space $(\mathbb C^m,0)$, we know that $U'=U\times (\mathbb C^m,0)$, where $U\subset M_H$ is an open subset. Recall that $I'_j(U')$ is given by
$$
I'_j(U')=\{f\circ\sigma;\; f\in I_j(U)\}\cdot \mathcal O_{M'}(U'),
$$
which is an ideal of the ring $O_{M'}(U')$. On the other hand, we define
$$
\mathbf J'_j(U')=\{f\circ\sigma;\; f\in \mathbf J_j(U)\}\cdot \operatorname{Int}\negthinspace\mathfrak E'(U').
$$
Since $\operatorname{Int}\negthinspace\mathfrak E'(U')=\{h\circ\sigma;\; h\in \operatorname{Int}\negthinspace\mathfrak E(U)\}$, we see directly that
$$
\mathbf J'_j(U')=\{f\circ\sigma;\; f\in \mathbf J_j(U)\}.
$$
In this way, we obtain the projecting system $\mathfrak S'$ for  $(\mathcal M',H')$, where $\mathcal M'$ is the transform of $\mathcal M$ by $\sigma$ and $H'=H\times(\mathbb C^m,0)$.

\subsection{Blowing-up Projecting Systems}
\label{ingl Transformacion de sistemas proyectantes por explosion de centros permitidos}
Let $Y\subset H$ be a permissible center for $\mathcal M$. Since $\mathcal M$ is reduced, we know that $Y$ has codimension greater than or equal to two in $M$, hence it has codimension greater than or equal to one in $H$. Let us consider the blowing-up
$$
\pi:(M',E')\rightarrow (M,E)
$$
centered at $Y$. As we know, the morphism $\pi$ induces the blowing-up
$$
\bar\pi:(H',E'\vert_{H'})\rightarrow (H,E\vert_H)
$$
centered at $Y$, where $H'$ is the strict transform of $H$ by $\pi$ and the divisor $E'\vert_{H'}$ coincides with  $\bar\pi^{-1}(E\vert_H\cup Y)$.

Let us define the {\em transform $\mathfrak S'$ of $\mathfrak S$ by $\pi$}. We put
$
	\mathfrak S'=(\mathfrak E', \mathfrak J')
$,
where $\mathfrak E'$ is the transform of $\mathfrak E$ by $\pi$. It remains to
 describe $\mathfrak J'=\{\mathbf J'_j\}_{j=1}^k$.
 Accordingly to Definition \ref{def:ingl sistemaproyectante}, we consider the morphism
$$
\sigma:(M'_{H'},E'_{H'})\rightarrow (M,E)
$$
given as composition of the open inclusion $M'_{H'}\subset M'$ with the blowing-up $\pi$.
Let $\mathcal J_{\pi^{-1}(Y)}$ be the ideal sheaf in $M'_{H'}$ defining the exceptional divisor $\pi^{-1}(Y)\cap M'_{H'}$. We denote
$$
{\boldsymbol{\mathcal J}}_{\pi^{-1}(Y)}=
{\mathcal J}_{\pi^{-1}(Y)}\cap \operatorname{Int}\negthinspace{\mathfrak E'}.
$$
Noting that $\mathcal J_{\pi^{-1}(Y),P'}$ is generated by a first integral of $\mathfrak E'$ at the points $P'$ in $M'_{H'}$, we see  that $\boldsymbol{\mathcal J}_{\pi^{-1}(Y)}$ is an $\mathfrak E'$-projectable generator of $\mathcal J_{\pi^{-1}(Y)}$. Consider now the $\mathfrak E$-projectable generator $\mathbf J_j$ of $I_j$ in $\mathfrak J$. Recall that
$$
\sigma^{-1}I_j=I_j'\mathcal J_{\pi^{-1}(Y)}^{d_j}.
$$
Denote by $\sigma^*\mathbf J_j$ the sheaf of $\operatorname{Int}(\mathfrak E')$-modules generated by $f\circ \sigma$, where $f$ varies over the sections of $\mathbf J_j$. In terms of germs, given a point $P'\in M'_{H'}$ and $P=\sigma(P')$, we have that
$$
(\sigma^*\mathbf J_j)_{P'}=\{f\circ\sigma;\; f\in \mathbf J_{j,P}\}\cdot \operatorname{Int}\negthinspace\mathfrak E'_{P'}.
$$
Then, we have that $\sigma^*\mathbf J_j$ is an $\mathfrak E'$-projectable generator of $\sigma^{-1}I_j$. Since $\sigma^{-1}I_j$ is divisible by $\mathcal J_{\pi^{-1}(Y)}^{d_j}$, then $\sigma^*\mathbf J_j$ is divisible by $\boldsymbol{\mathcal J}_{\pi^{-1}(Y)}^{d_j}$. In this way, we obtain an $\mathfrak E'$-projectable generator $\mathbf J'_j$ of $I'_j$ given by the relation
$$
\sigma^*\mathbf J_j=\mathbf J'_j\cdot \boldsymbol{\mathcal J}_{\pi^{-1}(Y)}^{d_j}.
$$
Thus, we obtain a projecting system $\mathfrak S'$ for $(\mathcal M', H')$, where $\mathcal M'$ is the transform of $\mathcal M$ by $\sigma$ and $H'$ is the strict transform of $H$.
\begin{remark}
\label{rk: coordinates}
Let us show how are described these objects in terms of coordinates. Fix two points $P\in Y$ and $P'\in H'\cap \pi^{-1}(P)$. We can choose a rectified projecting chart $(\mathfrak c=(U,\xi),\mathbf x,z)$ around $P$ such that
$$
Y\cap U=(z=0)\cap (x_1=x_2=\cdots=x_t=0).
$$
Moreover, we can assume that the blowing-up $\pi$ is given at $P'$ in coordinates $\mathbf x',z'$ by the relations $x_1=x'_1$, $z=x'_1z'$,
$$
x_s=x'_1(x'_s+\lambda_s), \quad  \lambda_s\in\mathbb C,\quad s=2,3,\ldots,t
$$
and $x_s=x_s'$, for $s=t+1,t+2,\ldots,n-1$, where $n=\dim M$. We know that
$\mathbf J_j=\{hF_j;\partial h/\partial z=0\}$, where $F_j$ generates $I_j$ at $P$. The ideal $I'_j$ is generated at $P'$ by
$$
F'_j= (x'_1)^{-d_j}(F_j\circ \sigma).
$$
The transformed projecting axis is given by $\partial/\partial z'$ in the coordinates $\mathbf x',z'$ and the $\mathfrak E'$-projectable generator of $I'_j$ is $\mathbf J'_j=\{h'F'_j;\;\partial h'/\partial z'=0\}$.
\end{remark}

\subsection{Projected Space of a Projecting System}
Let us consider a projecting system $\mathfrak S=(\mathfrak E,\mathfrak J)$  for $(\mathcal M,H)$. In this subsection we construct an idealistic space
$$
\mathfrak G_{\mathcal M, H}=\left(H,E\vert_H, \mathcal N\right),
\quad  \mathcal N=\{(N_{js},d_j-s)\}_{1\leq j\leq k,\, 0\leq s\leq d_j-1},
$$
of $(H,E\vert_H)$ that we call the {\em projected space of $\mathcal M$ by $\mathfrak S$ over $(M,E,H)$}.

Let us define the ideal sheaves $N_{js}\subset \mathcal O_H$. Take a projecting chart $\mathfrak c=(U,\xi)$ belonging to the projecting axis $\mathfrak E$. For each pair $j,s$, let
$$
\xi^{s}(\mathbf J_j)
$$
be
the iterated $s$-fold application of $\xi$ to the $\mathfrak E$-projectable generator $\mathbf J_j$ of $I_j$. Making a local computation for a section $hF_j$ of $\mathbf J_j$, we have that
$$
\xi(hF_j)=h\xi F_j, \quad h\in \operatorname{Int}\negthinspace\mathfrak E.
$$
We conclude that $\xi^{s}(\mathbf J_j)$ is an $\mathfrak E\vert_U$-projectable module. We define $N_{js}\vert_{U\cap H}$ to be the ideal sheaf of $\mathcal O_H\vert_{H\cap U}$ generated by the restriction
$$
\xi^{s}(\mathbf J_j)\vert_{H\cap U}
$$
of $\xi^{s}(\mathbf J_j)$ to $H\cap U$. This definition is compatible with the corresponding to the other projecting charts in the intersection of the domains. We define the ideal sheaves $N_{js}$ by a gluing procedure from the $N_{js}\vert_{U\cap H}$.

Before continuing with the properties of these objects, let us see how are the ideals $N_{js}$ in appropriated local coordinates. Consider a rectified projecting chart $(\mathfrak c=(U,\xi), \mathbf x,z)$ belonging to the axis $\mathfrak E$.  Assume also that the ideals $I_j\vert_U$ are generated by functions $F_j\in \mathcal O_M(U)$. We can write
\begin{equation}
\label{eq:ingl descomposicionparalaproyeccion}
\textstyle	F_j=\sum_{s=0}^\infty G_{js}(\mathbf x)z^{s}.
\end{equation}
Then, each ideal sheaf $N_{js}\vert_{H\cap U}$ is generated by $G_{js}(\mathbf x)$.

\begin{remark}
The assumption that the idealistic space $\mathcal M$ is reduced guaranties that not all the ideals $N_{js}$ are zero. Indeed, in terms of equations, if all that ideals are zero, we get that $z=0$ must be in the singular locus.
\end{remark}
A basic property of the projected space $\mathfrak S_{\mathcal M,H}$ is the following one:

\begin{proposition}
	\label{prop:ingl lugarsingular de la proyeccion}
	The singular locus of $\mathfrak S_{\mathcal M,H}$ is the restriction of the singular locus of $\mathcal M$ to $H$, that is
	$$
	\operatorname{Sing}(\mathfrak S_{\mathcal M,H})= H \cap \operatorname{Sing} \mathcal M.
	$$
\end{proposition}
\begin{proof}
	Follows directly by taking local equations and coordinates.
\end{proof}
\begin{remark}
\label{rk:centrospermitidos}
In the above situation, consider a non-singular closed analytic subset $Y\subset H$. Then we have that $Y$ has normal crossings with $E$ if and only if $Y$ has normal crossings with $E\vert_H$, inside $H$. In the case that $(M,E,H)$ is a transverse hypersurface, the observation is straightforward. If $(M,E,H)$ is not transverse, then $H$ coincides locally with an irreducible component of $E$, that is, we locally have that $E=E^*\cup H$, the normal crossings property between $Y$ and $E\vert_H=E^*\cap H$ and the fact that $Y\subset H$ assure the normal crossings property with $E$.

As a consequence, the permissible centers for $\mathfrak S_{\mathcal M,H}$ coincide with the permissible centers for $\mathcal M$ that are contained in $H$.
\end{remark}

\section{Commutativity and Equivalence}
\label{Resultados de conmutatividad y equivalencia}

Assume that we have a projecting system $\mathfrak S=(\mathfrak E,\mathfrak J)$  for $(\mathcal M,H)$. Consider a morphism
$$
\sigma:(M',E')\rightarrow (M,E)
$$
that is either an open projection or the blowing-up of $(M,E)$ with center $Y$ permissible for $\mathcal M$ and such that $Y\subset H$. Since $\mathcal M$ is adjusted-reduced, the codimension of $Y$ in $M$ is greater than or equal to two, hence its codimension in $H$ is greater than or equal to one. In both cases, we have induced a morphism
$$
\bar\sigma:(H',E'\vert_{H'})\rightarrow (H,E\vert_H),
$$
that is an open projection or a blowing-up centered at $Y$, respectively. Note that in the case of a blowing-up, the center $Y$ is also permissible for the projected space $\mathfrak S_{\mathcal M,H}$, in view of Proposition \ref{prop:ingl lugarsingular de la proyeccion}.

Let $\mathfrak S'$ be the transform of $\mathfrak S$ by $\sigma$ and let $(\mathfrak S_{\mathcal M,H})'$ be the transform by $\bar \sigma$ of the projected space $\mathfrak S_{\mathcal M,H}$. Denote by $\mathfrak G'$, $\mathcal M'$ and $H'$ the transform of $\mathfrak G$, $\mathcal M$ and the strict transform of $H$ by $\sigma$, respectively.
\begin{proposition}
\label{prop: commutativityofprojections}
	We have that $(\mathfrak S_{\mathcal M,H})'=\mathfrak S'_{\mathcal M',H'}$.
\end{proposition}
 \begin{proof}
 We use the notations and computations done in Subsection \ref{Projecting systems}.

Let $V$ be a non-empty subset of $M_H$. The commutativity property for the restriction to $V$ is written as
\begin{equation}
	\label{eq:ingl conmutatividadderestriccionaunabierto}
\mathfrak S_{\mathcal M,H}\vert_V=(\mathfrak S\vert_V)_{\mathcal M\vert_V,H\cap V}.
\end{equation}
It follows directly from the definitions. In the case that $\sigma$ is a projection over the first factor, the commutativity property is also deduced automatically from the definitions.

Assume that $\sigma$ is a  permissible blowing-up for $\mathcal M$, centered at $Y \subset H$. The commutativity property can be checked in local coordinates. Let us fix two points $P\in Y$ and $P'\in H'\cap \pi^{-1}(P)$, a rectified projecting chart $(\mathfrak c=(U,\xi),\mathbf x,z)$ belonging to $\mathfrak E$, centered at $P$, and such that
$$
Y\cap U=(z=0)\cap (x_1=x_2=\cdots=x_t=0).
$$
Moreover, let us assume without loss of generality that $\sigma$ is given at $P'$ in coordinates $\mathbf x',z'$ as in Remark \ref{rk: coordinates}.

We know that
$\mathbf J_j=\{hF_j;\; \partial h/\partial z=0\}$, where $F_j$ generates $I_j$. Write
$
\textstyle F_j=\sum_{s=0}^\infty G_{sj}(\mathbf x)z^s
$.
Since $Y$ belongs to the singular locus of $\mathcal M$, we can write
$$
G_{sj}(\mathbf x)=(x'_1)^{d_j-s}G'_{sj}(\mathbf x'),
$$
for the indices $s\leq d_j$. The ideal $I'_j$ is generated by
$$
\textstyle F'_j= (x'_1)^{-d_j}F_j(\mathbf x,z)=\sum_{s=0}^{d_j-1}G'_{sj}(\mathbf x')(z')^{s}+(z')^{d_j}\tilde{F}_j.
$$
The transform $\mathfrak c'$ of the chart $\mathfrak c$ belongs to the transformed axis $\mathfrak E'$ and is given by $\partial/\partial z'$ in the coordinates $\mathbf x',z'$. The $\mathfrak E'$-projectable generator of $I'_j$ is $\mathbf J'_j=\{h'F'_j;\;\partial h'/\partial z'=0\}$, locally at $P'$.
The projected space $\mathfrak S_{\mathcal M,H}$ is given (locally) by the list
$$
\mathcal N=\{(G_{sj}\mathcal O_{H,P},d_j-s)\}_{0\leq j\leq k, 0 \leq s\leq d_j-1},
$$
and the projected space of $\mathfrak S'_{\mathcal M',H'}$ coincides with the transform of $\mathcal N$ and it is given by the list
$$
\mathcal N'=\{(G'_{sj}\mathcal O_{H',P'},d_j-s)\}_{0\leq j\leq k, 0 \leq s\leq d_j-1}.
$$
Thus, we obtain the desired commutativity property.
\end{proof}

\subsection{Basic Properties of the Projections}
\label{Basic Properties of the Projections} Let us summarize here the results in Proposition
\ref{prop:ingl lugarsingular de la proyeccion}, Proposition \ref{prop: commutativityofprojections}  and
Remark
\ref{rk:centrospermitidos}, stated for any adjusted and reduced idealistic space $\mathcal M$ of $(M,E)$ and any hypersurface $(M,E,H)$:
\begin{itemize}
	\item $\operatorname{Sing}(\mathfrak S_{\mathcal M,H})=H\cap \operatorname{Sing \mathcal M}$.
	\item The permissible centers for $\mathcal M$ contained in $H$ coincide with the permissible centers of $\mathfrak S_{\mathcal M,H}$.
	\item  $(\mathfrak S_{\mathcal M,H})'=\mathfrak S'_{\mathcal M',H'}$, for the transforms under open projections and blowing-up with permissible centers.
\end{itemize}
 As a first consequence of these properties we obtain the following equivalence result:

\begin{proposition}
\label{prop:equivalenciadeproyecciones}
	Assume that $\mathcal M_\alpha$ and $\mathcal M_\beta$ are two equivalent idealistic spaces over $(M,E)$ and let $\mathfrak S^\alpha$ and $\mathfrak S^\beta$ be two projecting systems for $(\mathcal M_\alpha,H)$ and $(\mathcal M_\beta,H)$, respectively.
Then $\mathfrak S^\alpha_{\mathcal M_\alpha,H}$ and $\mathfrak S^\beta_{\mathcal M_\beta,H}$ are equivalent  idealistic spaces over $(H,E\vert_H)$.
\end{proposition}

\subsection{Projection of Idealistic Exponents}
Consider an adjusted and reduced idealistic exponent $\mathcal E$ over $(M,E)$.
In this subsection we construct an idealistic exponent
$\operatorname{pr}_{H}\mathcal E$
 over $(H,E\vert_H)$, that we call the {\em projected idealistic exponent of $\mathcal E$ over $(M,E,H)$}, satisfying the three properties stated in the beginning of this Part \ref{Projections of Idealistic Exponents}.

Take a point $P\in H$. By using suitable equations we can construct a rectifiable projecting chart
$$
\mathfrak c_{P}=(U_P,\xi_P)
$$
for $(M,E,H)$, where $P\in U_P\subset M_{H}$. Reducing the size of $U_P$, we also find an idealistic chart $\mathcal U_P=(U_P,E\cap U_P,\mathcal L_P)$ of $\mathcal E$ in such a way that the marked ideals of $\mathcal L_P$ are generated by global functions that have a decomposition as in Equation \eqref{eq:ingl descomposicionparalaproyeccion}. Finally, using the equations, we obtain a projecting system
$
\mathfrak S^P=(\mathfrak E_P,\mathfrak J_P)
$
for $(\mathcal U_P, H\cap U_P)$. Choosing a finite open covering of $M_H$ by open subsets $\{U_\alpha\}_{\alpha\in \Lambda}$, among the $U_P$, we get a finite family of projecting systems
$
\mathfrak S^\alpha
$
for $(\mathcal U_\alpha, H_\alpha)$,
where $H_\alpha=H\cap U_\alpha$ and each $\mathcal U_\alpha$ is an idealistic chart of $\mathcal E$, defined in $U_\alpha$.

Given two indices $\alpha,\beta\in \Lambda$, we have that
$
\mathcal U_{\alpha\beta}$
and
$
\mathcal U_{\beta\alpha}
$
are equivalent, since $\mathcal U_\alpha$ and $\mathcal U_\beta$ belong to $\mathcal E$ (we take the usual notations). Consider now the collection of idealistic charts
$$
\mathfrak S^\alpha_{\mathcal U_\alpha,H_\alpha}=\left(H_\alpha,E_\alpha\vert_{H_\alpha},\mathcal N_\alpha\right),\quad \alpha\in \Lambda
$$
In view of Proposition
\ref{prop:equivalenciadeproyecciones}, we conclude that the family $\{\mathfrak S^\alpha_{\mathcal U_\alpha,H_\alpha}\}_{\alpha\in\Lambda}$ is an idealistic atlas over $(H,E\vert_H)$. Even more, any other idealistic atlas obtained by this procedure is equivalent to it. In this way we define without ambiguity the projected idealistic exponent  $\operatorname{pr}_H\mathcal E$ of $\mathcal E$ over $(M,E,H)$.

We obtain the three required properties from the case of a single idealistic space presented in Subsection \ref{Basic Properties of the Projections}.

\section{Projections of $e$-Flowers  over the Divisor}
\label{Projection of an idealistic immersed exp-atlas over a component of the divisor}

Let $(M,E)$ be an ambient space and let $F$ be an irreducible component of the divisor $E$, thus,  we have a hypersurface  $(M,E,F)$.
Consider an adjusted and reduced $e$-flower $\mathcal F$ over $(M,E)$. In this section we build the {\em projection $\operatorname{pr}_F\mathcal F$}. It will be an idealistic $(e-1)$-flower over
$(F,E\vert_F)$, satisfying the three basic properties:
\begin{itemize}
\item $\operatorname{Sing}(\operatorname{pr}_F{\mathcal F})= F\cap \operatorname{Sing}{\mathcal F}$.
    \item The permissible centers of $\mathcal F$ contained in $F$ are exactly the permissible centers of $\operatorname{pr}_F{\mathcal F}$.
    \item We have the commutativity property $\operatorname{pr}_{F'}({\mathcal F'})=(\operatorname{pr}_F{\mathcal F})'$ for the transforms under an open projection or a blowing-up with a permissible center contained in $F$.
\end{itemize}
The uniqueness is assured. The existence comes from the case of idealistic exponents as we detail now.
Let us consider an immersed exp-idealistic $e$-atlas $\mathcal Q$ belonging to $\mathcal F$ given by
$$
\mathcal Q=\{
\mathcal W_\alpha
\}_{\alpha\in \Lambda},
\quad \mathcal W_\alpha=(M_\alpha,E_\alpha,N_\alpha,\mathcal E_\alpha).
$$
Take an index $\alpha\in \Lambda$, put  $F_\alpha=F\cap M_\alpha$. Recall that $N_\alpha$ is transverse to $F$ and hence
$$
(N_\alpha, E_\alpha\vert_{N_\alpha},F_\alpha\cap N_\alpha),\quad E_\alpha\vert_{N_\alpha}=E_\alpha\cap N_\alpha,
$$
is a hypersurface of $(N_\alpha,E_\alpha\vert_{N_\alpha})$. Note that $F_\alpha\cap N_\alpha$ is a disjoint union of irreducible components of $E_\alpha\vert_{N_\alpha}$. Recalling that $\mathcal Q$ is adjusted and reduced, we can project $\mathcal E_\alpha$ over $(N_\alpha, E_\alpha\vert_{N_\alpha},F_\alpha\cap N_\alpha)$ to obtain an $(e-1)$-dimensional idealistic exponent
$\operatorname{pr}_{F_\alpha\cap N_\alpha}{\mathcal E}_{\alpha}$ over
$(F_\alpha\cap N_\alpha, E_\alpha\vert_{F_\alpha\cap N_\alpha})$ and hence an immersed $(e-1)$-dimensional idealistic exponent
$$
\widetilde{\mathcal W}_\alpha=(F_\alpha, E_\alpha\vert_{F_\alpha}, F_\alpha\cap N_\alpha,\operatorname{pr}_{F_\alpha\cap N_\alpha}{\mathcal E}_{\alpha} ).
$$
Consider the following proposition, that can be proved by a systematic use of the three basic properties of the projections of idealistic exponents:
\begin{proposition}
\label{prop:porjectionofimmersedexpatlas}
The family $\widetilde{\mathcal Q}=\{\widetilde{\mathcal W}_\alpha\}_{\alpha\in \Lambda}$ is an $(e-1)$-dimensional immersed exp-idealistic atlas over the ambient space $(F,E\vert_F)$.
Moreover, the following properties hold:
\begin{enumerate}
	\item $\operatorname{Sing}(\widetilde{\mathcal Q})=F\cap \operatorname{Sing}({\mathcal Q})$.
	\item The permissible centers for $\widetilde{\mathcal Q}$ are exactly the  permissible centers  for $\mathcal Q$ contained in $F$.
	\item We have the commutativity $
	\widetilde{\mathcal Q'}=\widetilde{\mathcal Q}'$ under open projections and blowing-ps with permissible centers contained in $F$.
	\end{enumerate}
\end{proposition}
Now, it suffices to define $\operatorname{pr}_F\mathcal F$ to be the idealistic $(e-1)$-flower  over $(F,E\vert_F)$ given by the atlas $\widetilde{\mathcal Q}$.
\strut
\vfill
\pagebreak

\part{Maximal Contact}
In this part, we end the proof of Theorem \ref{th: redsingimmersedidealisticflowers}. In view of the results in the previous parts, it is enough to prove the statement  corresponding to the adjusted-reduced case, presented in Equation
\eqref{eq:ingl contacto maximal}  of Part
\ref{Guide for the Reduction of Singularities}. More precisely, we have to prove the following statement:
\begin{proposition}
	\label{prop: adjustedreducedcase}
	Assume that all the idealistic $(e-1)$-flowers have reduction of singularities. Consider an adjusted and reduced idealistic $e$-flower $\mathcal F$ over $(M,E)$. Then $\mathcal F$ has reduction of singularities.
\end{proposition}

We start by recalling the following definition of {\em maximal contact} in terms of idealistic flowers:

\begin{definition}
	\label{def:contactomaximalflowers}
	Let $\mathcal F$ be an idealistic $e$-flower over an ambient space $(M,E)$. We say that an idealistic $(e-1)$-flower $\mathcal H$ over $(M,E)$ has  {\em maximal contact with $\mathcal F$} if $\mathcal F$ and $\mathcal H$ are equivalent as idealistic flowers.
\end{definition}

Let us also recall that being equivalent means that the two idealistic flowers $\mathcal F$ and $\mathcal H$ have the same permissible test systems.

\begin{remark}
	\label{rk: florescontactomaximal}
	The following ones are direct consequences of the definition: \begin{enumerate}
		\item  If $\mathcal H_1$ and $\mathcal H_2$ have maximal contact with $\mathcal F$, then $\mathcal H_1=\mathcal H_2$.
		\item If there is $\mathcal H$ having maximal contact with $\mathcal F$, then $\mathcal F$ is adjusted and reduced. It is reduced since the codimension of the singular locus is greater than or equal to $e-2$. Let us see that it is adjusted. If the order at a point is greater than one, we can find a permissible test system that produces a codimension $n-e+1$ singular locus for $\mathcal F$, hence for $\mathcal H$, by a curve-divisor procedure.
		\item Any reduction of singularities of $\mathcal H$ induces a reduction of singularities of $\mathcal F$.
	\end{enumerate}
\end{remark}

Thus, in order to obtain a proof of Proposition \ref{prop: adjustedreducedcase}, and hence a proof of Theorem \ref{th: redsingimmersedidealisticflowers}, it is enough to prove the following statement:

\begin{proposition}
	\label{prop: maximalcontact}
	Assume that all the idealistic $(e-1)$-flowers have reduction of singularities.
	Consider an adjusted and reduced idealistic $e$-flower $\mathcal F$ over $(M,E)$. We can find a morphism
	$$
	\sigma:(M',E')\rightarrow (M,E),
	$$
	composition of a finite sequence of permissible blowing-ups such that there is an idea\-listic $(e-1)$-flower $\mathcal H'$ over $(M',E')$  having maximal  contact with  the transform $\mathcal F'$ of $\mathcal F$ by $\sigma$.
\end{proposition}

We are going to prove Proposition \ref{prop: maximalcontact} in several steps.
The first step is to show how to separate the ``old components of the divisor'' from the singular locus. The second step is to find ``maximal contact hypersurfaces'', obtained from Tschirnhaus' coordinate changes, when we have an empty divisor. In this way, we are done in the special case when  $E=\emptyset$. Thanks to the first step, after finitely many permissible blowing-ups, we eliminate the old components and using the hypersurfaces obtained for the case of an empty divisor, we get the idealistic $(e-1)$-flower $\mathcal H'$ that gives maximal contact with $\mathcal F'$.

\section{Separating Old Components}
Here we  separate ``old components'' of the divisor from the singular locus. To manage the idea of ``old component'', we consider {\em splittings}
$$
E=E^*\cup D,
$$
where both $E^*$ and $D$ are unions of disjoint sets of irreducible components of $E$. The divisor $E^*$ will stand for the ``old components''.

Assume that $\pi:(M',E')\to (M,E)$ is the blowing-up of $(M,E)$ with a center $Y$ having normal crossings with $E$. We know that
$$
E'=\pi^{-1}(E\cup Y).
$$
The \emph{transformed splitting} $E'={E'}^* \cup D'$ of $E=E^*\cup D$
is given by
taking ${E'}^*$ to be the strict transform of $E^*$ by $\pi$ and $D'=\pi^{-1}(D\cup Y)$.

If we have an open projection $\sigma:(M',E')\rightarrow (M,E)$, the {\em transformed splitting} is
$
E'={E'}^*\cup D'
$,
where ${E'}^*=\sigma^{-1}(E^*)$ and $D'=\sigma^{-1}(D)$.

\begin{proposition}
	\label{prop:separer viejas components}
	Consider an adjusted and reduced idealistic $e$-flower $\mathcal F$ over $(M,E)$. Take a splitting $E=E^*\cup D$. Assume that any idealistic $(e-1)$-flower has reduction of singularities.
	There is a composition $\sigma:(M',E')\rightarrow (M,E)$ of a finite sequence of permissible blowing-ups such that
	$$
	{E'}^*\cap \operatorname{Sing}\mathcal F'=\emptyset,
	$$
	where $\mathcal F'$ is the transform of $\mathcal F$ by $\sigma$ and $E'={E'}^*\cup D'$ is the transformed splitting of $E=E^*\cup D$ by $\sigma$.
\end{proposition}
\begin{proof}
It is enough to deal with the case when $E^*=F$ is a single component of the exceptional divisor. Let $\operatorname{pr}_F\mathcal F$ be the projection of $\mathcal F$ over $F$ as constructed in
Section
\ref{Projection of an idealistic immersed exp-atlas over a component of the divisor}. Thanks to  our hypothesis on the existence of reduction of singularities for idealistic $(e-1)$-flowers, we can take a reduction of singularities of $\operatorname{pr}_F\mathcal F$.
In view of the three properties stated in Section
\ref{Projection of an idealistic immersed exp-atlas over a component of the divisor}, this reduction of singularities allows us to obtain the situation $F'\cap \operatorname{Sing}\mathcal F'=\emptyset$, as desired.
\end{proof}

\section{Maximal Contact Hypersurfaces}

Maximal contact hypersurfaces are given by the following definition:
\begin{definition}
	\label{def:ingl contactomaximal}
	Let $\mathcal E$ be an idealistic exponent over $(M,E)$ and let $(M,E,H)$ be a tranverse hypersurface of $(M,E)$. We say that $(M,E,H)$ has {\em ma\-ximal contact with $\mathcal E$} if for any $\mathcal E$-permissible test system $\mathcal S$, we have the following properties:
	\begin{enumerate}
		\item The singular locus $\operatorname{Sing}\mathcal E'$ of the transform $\mathcal E'$ of $\mathcal E$ by $\mathcal S$ has codimension greater than or equal to two.
		\item $\operatorname{Sing}\mathcal E'\subset H'$, where $H'$ is the strict transform of $H$ by $\mathcal S$.
	\end{enumerate}
\end{definition}
Note that the strict transform $(M',E',H')$ of $(M,E,H)$ under the $\mathcal E$-permissible test system $\mathcal S$ has also maximal contact with $\mathcal E'$.
\begin{remark}
By a similar argument to the one in Remark \ref{rk: florescontactomaximal}, we have that the idealistic exponent $\mathcal E$ is necessarily adjusted and reduced, otherwise there is no maximal contact hypersurface.
\end{remark}
\begin{remark}
\label{rk:expidealisticodelcontactomaximal}
Let $\mathcal E$ be an idealistic exponent over an $n$-dimensional ambient space $(M,E)$. Denote by $\mathcal F$ the idealistic $n$-flower over $(M,E)$ defined by the immersed exp-idealistic $n$-chart
	$$
	\mathcal W=(M,E, M,\mathcal E).
	$$
	Assume that $(M,E,H)$ is a hypersurface of $(M,E)$ having maximal contact with $\mathcal E$. We can project $\mathcal E$ over $(M,E,H)$ to obtain an idealistic exponent
	$$
	\widetilde{\mathcal E}=\operatorname{pr}_{H}(\mathcal E)
	$$
	over $(H,E\vert_H)$. This gives to us an immersed exp-idealistic $(n-1)$-chart
	$$
	\widetilde{\mathcal W}=(M,E, H,\widetilde{\mathcal E})
	$$
	that defines an idealistic $(n-1)$-flower $\widetilde{\mathcal F}$ over $(M,E)$. Then  $\mathcal F$ and $\widetilde{\mathcal F}$ are equivalent idealistic flowers. In other words, we have that $\widetilde{\mathcal F}$ has maximal contact with $\mathcal F$, see Definition \ref{def:contactomaximalflowers}.
\end{remark}

\begin{remark}
\label{rk:tschirnhaus}
	The following one is the basic example in the Theory of Maximal Contact. Take $M=(\mathbb C^n,0)$, with coordinates $\mathbf x,z$, such that  $E\subset\left(\prod_{i=1}^{n-1}x_i=0\right)$. Assume that the adjusted and reduced idealistic exponent $\mathcal E$ contains an idealistic chart with a list having the marked ideal
	$
	(f\mathcal O_{\mathbb C^n,0}, d),
	$
	where $\nu_0f=d$ and $f$
	is written as
	\begin{equation}
\label{eq:tschirnhaus}
	f=z^d+\sum_{i=2}^dg_i(\mathbf x)z^{d-i}.
	\end{equation}
	The hypersurface of maximal contact is $H=(z=0)$. The proof of this statement is founded in the classical behaviour of the Tschirnhaus form of $f$ given in Equation \eqref{eq:tschirnhaus}. More precisely, one sees immediately that the singular locus is contained in $z=0$; moreover, the Tschirnhaus form of $f$ is stable under open projections and permissible blowing-ups at the points of the singular locus.
\end{remark}

\subsection{Maximal Contact Without Divisor}In this subsection, we recall the basic result in the Theory of Maximal Contact:

\begin{proposition}
	\label{prop:maximalcontact sin divisor}
	Let $\mathcal E$ be an adjusted and reduced idealistic exponent over the ambient space $(M,\emptyset)$.  For any $P\in \operatorname{Sing}\mathcal E$, there are an open set $U\subset M$ with $P\in U$ and a closed hypersurface $(U,\emptyset,H)$ having maximal contact with $\mathcal E\vert_U$.
\end{proposition}
\begin{proof} Up to take a smaller open subset if necessary, we can assume that there is an idealistic chart belonging to $\mathcal E$ of the form
	$$
	(M,\emptyset,\mathcal L),\quad \mathcal L=\{(I,d)\}\cup \{(I_j,d_j)\}_{j=2}^k.
	$$
	such that $\nu_PI=d$. In view of Weierstrass Preparation Theorem, we can choose local coordinates $\mathbf x,y$ around $P$ and a generator $f$ of $I_P\subset \mathcal O_{M,P}$ having the form
	$$
	f=y^d+\tilde g_1(\mathbf x)y^{d-1}+\tilde g_2(\mathbf x)y^{d-2}+\cdots+\tilde g_{d}(\mathbf x).
	$$
	Note that $\nu_P(\tilde g_i)\geq i$, for $i=1,2,\ldots,d$. Let us perform the coordinate change
	$$
	z=y+(1/d)\tilde g_1(\mathbf x)\quad  \text{(Tschirnhaus)}.
	$$
	Then, we write $f$ as $f=z^d+g_2(\mathbf x)z^{d-2}+\cdots+g_d(\mathbf x)$.
	Taking $H=(z=0)$ we obtain the maximal contact property, in view of Remark \ref{rk:tschirnhaus}.
\end{proof}
\subsection{Systems of Maximal Contact Hypersurfaces} Let $(M,E)$ be an ambient space and let us consider a splitting $E=E^*\cup D$ of  $E$ into two normal crossings divisors without common irreducible components. Let us consider an adjusted and reduced idealistic $e$-flower  $\mathcal F$ over $(M,E)$.

A {\em system $\mathfrak H$ of maximal contact hypersurfaces  for $\mathcal F$ associated to the splitting $E=E^*\cup D$} is a finite family
$$
\mathfrak H=\{(\mathcal V_\alpha, H_\alpha,D_\alpha)\}_{\alpha\in \Lambda}, \quad \mathcal V_\alpha=(M_\alpha, E_\alpha,N_\alpha,\mathcal L_\alpha),
$$
satisfying the following properties:
\begin{enumerate}
	\item The family $\mathcal P=\{\mathcal V_\alpha\}_{\alpha\in \Lambda}$ is an immersed idealistic $e$-atlas over $(M,E)$ belonging to $\mathcal F$.
	\item For any $\alpha\in \Lambda$, we have that
$D_\alpha=D\cap M_\alpha$ and thus there is a splitting  $E_\alpha=E_\alpha^*\cup D_\alpha$, induced by $E=E^*\cup D$.
	\item Each  $(N_\alpha, D_\alpha\vert_{N_\alpha}, H_\alpha)$ is a maximal contact hypersurface for the idealistic space $(N_\alpha, D_\alpha\vert_{N_\alpha}, \mathcal L_\alpha)$. (Let us recall that we have $D_\alpha\vert_{N_\alpha}=D_\alpha\cap N_\alpha$, since $(M_\alpha, D_\alpha, N_\alpha)$ is a transverse ambient subspace of $(M_\alpha, D_\alpha)$).
\end{enumerate}
The systems of maximal contact hypersurfaces for $\mathcal F$ associated to a splitting may be transformed by $\mathcal F$-permissible test systems in a natural way, to obtain a new system of maximal contact hypersurfaces for the transform of $\mathcal F$, associated to the transformed splitting. The maximal contact hypersurfaces are transformed by taking the strict transform and the immersed idealistic $e$-atlases are transformed as we have already seen in Section  \ref{Immersed Idealistic Spaces}.

\begin{remark}
Note that we take immersed idealistic spaces $\mathcal V_\alpha$ instead of immersed idealistic exponents. The reason is that an idealistic exponent over $(M,E)$ does not define an idealistic exponent over $(M,D)$. Nevertheless, an immersed idealistic space over $(M,E)$ does define an immersed idealistic space over $(M,D)$.
\end{remark}
\begin{proposition}
	\label{prop:maximalcontacwithdivisor}
	Let $\mathfrak H$ be a system of maximal contact hypersurfaces for  an idealistic $e$-flower $\mathcal F$ over $(M,E)$ associated to the splitting
	$$
	E=E^*\cup D, \quad E^*=\emptyset, \quad D=E.
	$$
	Let $\mathcal E_\alpha$ be the idealistic exponent over $(N_\alpha, E_\alpha|_{N_\alpha})$  defined by $\mathcal L_\alpha$ and let $\widetilde{\mathcal E}_\alpha$ be the projection of $\mathcal E_\alpha$ over the hypersurface $(N_\alpha, E_\alpha|_{N_\alpha}, {H_\alpha})$. Then, the family
	$$
	\mathcal Q_\mathfrak H=\{\widetilde{\mathcal W}_\alpha=(M_\alpha, E_\alpha, H_\alpha,\widetilde{\mathcal E}_\alpha)\}_{\alpha\in \Lambda}
	$$
	is an immersed exp-idealistic $(e-1)$-atlas over $(M,E)$ that defines an idealistic $(e-1)$-flower $\mathcal H$ over $(M,E)$ having maximal contact with $\mathcal F$.
\end{proposition}
\begin{proof}
	We know that the $e$-flower $\mathcal F$ is described by the immersed exp-idealistic $e$-atlas $\mathcal Q=\{\mathcal W_\alpha\}_{\alpha \in \Lambda}$, where
	$
	\mathcal W_\alpha=(M_\alpha, E_\alpha, N_\alpha,\mathcal E_\alpha).
	$
	In particular, we have the equivalence
	$$
	\mathcal W_{\alpha}|_{M_{\alpha\beta}} \sim \mathcal W_{\beta}|_{M_{\alpha\beta}}.
	$$
	In view of Remark \ref{rk:expidealisticodelcontactomaximal} and Subsection \ref{subsec: change of codimension}, we know that
	$$
	\mathcal W_{\alpha} \sim \widetilde{\mathcal W}_{\alpha}, \quad \mathcal W_{\beta} \sim \widetilde{\mathcal W}_{\beta}.
	$$
	Making the restriction to $M_{\alpha\beta}$ we conclude that $
	\widetilde{\mathcal W}_{\alpha}|_{M_{\alpha\beta}} \sim \widetilde{\mathcal W}_{\beta}|_{M_{\alpha\beta}}$. Hence $\mathcal Q_{\mathfrak H}$ is an immersed exp-idealistic $(e-1)$-atlas over $(M,E)$. Moreover, the
	equivalences $\mathcal W_{\alpha} \sim \widetilde{\mathcal W}_{\alpha}$ imply that $\mathcal F$ is equivalent to $\mathcal H$.
\end{proof}
\section{Conclusion}
Here we prove Proposition \ref{prop: maximalcontact}. This ends the proof of Theorem \ref{th: redsingimmersedidealisticflowers}.

Recall that we work under the induction assumption that the idealistic $(e-1)$-flowers have reduction of singularities. Consider an adjusted and reduced  idealistic $e$-flower $\mathcal F$  over an ambient space $(M,E)$.

By Proposition \ref{prop:maximalcontact sin divisor}, there is a system $\mathfrak H$ of maximal contact hypersurfaces  for $\mathcal F$ associated to the splitting $E=E^*\cup D$, with $D=\emptyset$.

By Proposition \ref{prop:separer viejas components}, there is a composition
$
\sigma:(M',E')\rightarrow (M,E)
$
of a finite sequence of $\mathcal F$-permissible blowing-ups such that
$$
{E'}^*\cap \operatorname{Sing}\mathcal F'=\emptyset.
$$
Take an open subset $U\subset M'$ containing the singular locus of $\mathcal F'$ such that $U\cap {E'}^*=\emptyset$.  Now, finding a maximal contact $(e-1)$-flower $\mathcal H'$ for $\mathcal F'$ is the same problem as finding such a maximal contact $(e-1)$-flower for $\mathcal F'|_U$.  Let $\mathfrak H'$ be transformed of $\mathfrak H$ by $\sigma$ and consider the restriction $\mathfrak H'|_U$. Since $U\cap {E'}^*=\emptyset$, we have that $E'\cap U=D'\cap U$. In this way, we are in the situation of Proposition \ref{prop:maximalcontacwithdivisor}, that provides the desired idealistic $(e-1)$-flower of maximal contact with $\mathcal F'|_U$. This ends the proof.

\bibliographystyle{plain}

\begin{thebibliography}{10}
	
	\bibitem{Abr-J}
	D.~Abramovich and A.~J. de~Jong.
	\newblock Smoothness, semistability, and toroidal geometry.
	\newblock {\em J. Algebraic Geom.}, 6(4):789--801, 1997.
	
	\bibitem{Abra-T-W}
	D.~Abramovich, M.~Temkin, and J.~W{\l}odarczyk.
	\newblock Functorial embedded resolution via weighted blowings up.
	\newblock {\em arXiv preprint:1906.07106}, 2019.
	
	\bibitem{Aro-H-V1}
	J.~M. Aroca, H.~Hironaka, and J.~L. Vicente.
	\newblock The theory of maximal contact.
	\newblock {\em Memorias de Matem{\' a}tica del Instituto "Jorge Juan"}, 29,
	1975.
	
	\bibitem{Aro-H-V2}
	J.~M. Aroca, H.~Hironaka, and J.~L. Vicente.
	\newblock Desingularization theorems.
	\newblock {\em Memorias de Matematica del Instituto "Jorge Juan"}, 30, 1977.
	
	\bibitem{Aro-H-V}
	J.~M. Aroca, H.~Hironaka, and J.~L. Vicente.
	\newblock {\em Complex analytic desingularization}.
	\newblock Springer, Tokyo, 2018.
	
	\bibitem{Bie-M}
	E.~Bierstone and P.~Milman.
	\newblock Canonical desingularization in characteristic zero by blowing up the
	maximum strata of a local invariant.
	\newblock {\em Invent. Math.}, 128(2):207--302, 1997.
	
	\bibitem{Can3}
	F.~Cano.
	\newblock Reduction of the singularities of codimension one singular foliations
	in dimension three.
	\newblock {\em Ann. of Math. (2)}, 160(3):907--1011, 2004.
	
	\bibitem{Can-F}
	F.~Cano and M.~Fern\'{a}ndez-Duque.
	\newblock Truncated local uniformization of formal integrable differential
	forms.
	\newblock {\em Qual. Theory Dyn. Syst.}, 21(1):Paper No. 18, 74, 2022.
	
	\bibitem{Can-R-S}
	F.~Cano, C.~Roche, and M.~Spivakovsky.
	\newblock Reduction of singularities of three-dimensional line foliations.
	\newblock {\em Rev. R. Acad. Cienc. Exactas F\'{\i}s. Nat. Ser. A Mat. RACSAM},
	108(1):221--258, 2014.
	
	\bibitem{Cos3}
	V.~Cossart.
	\newblock D{\'e}singularisation en dimension 3 et caract{\'e}ristique p.
	\newblock In {\em Algebraic Geometry and Singularities}, pages 3--7. Springer,
	1996.
	
	\bibitem{Cos-P2}
	V.~Cossart and O.~Piltant.
	\newblock Resolution of singularities of threefolds in positive characteristic
	ii.
	\newblock {\em Journal of Algebra}, 321(7):1836--1976, 2009.
	
	\bibitem{Cut}
	S.~D. Cutkosky.
	\newblock {\em Resolution of singularities}, volume~63.
	\newblock American Mathematical Soc., 2004.
	
	\bibitem{Cut2}
	S.~D. Cutkosky.
	\newblock \'{E}toiles and valuations.
	\newblock {\em J. Pure Appl. Algebra}, 221(3):588--610, 2017.
	
	\bibitem{Enc-H}
	S.~Encinas and H.~Hauser.
	\newblock Strong resolution of singularities in characteristic zero.
	\newblock {\em Comment. Math. Helv.}, 77(4):821--845, 2002.
	
	\bibitem{Gir1}
	J.~Giraud.
	\newblock Sur la th\'{e}orie du contact maximal.
	\newblock {\em Math. Z.}, 137:285--310, 1974.
	
	\bibitem{Gir2}
	J.~Giraud.
	\newblock R\'{e}solution des singularit\'{e}s (d'apr\`es {H}eisuke {H}ironaka).
	\newblock In {\em S\'{e}minaire {B}ourbaki, {V}ol. 10}, pages Exp. No. 320,
	101--113. Soc. Math. France, Paris, 1995.
	
	\bibitem{Hau}
	H.~Hauser.
	\newblock The {H}ironaka theorem on resolution of singularities (or: {A} proof
	we always wanted to understand).
	\newblock {\em Bull. Amer. Math. Soc. (N.S.)}, 40(3):323--403, 2003.
	
	\bibitem{Hir6}
	H.~Hironaka.
	\newblock La vo\^{u}te \'{e}toil\'{e}e.
	\newblock In {\em Singularit\'{e}s \`a {C}arg\`ese ({R}encontre
		{S}ingularit\'{e}s en {G}\'{e}om. {A}nal., {I}nst. \'{E}tudes {S}ci.,
		{C}arg\`ese, 1972)}, volume Nos. 7 et 8 of {\em Ast\'{e}risque}, pages pp
	415--440.
	
	\bibitem{Hir7}
	H.~Hironaka.
	\newblock {\em Introduction to the theory of infinitely near singular points},
	volume No. 28.
	\newblock Consejo Superior de Investigaciones Cient\'{\i}ficas, Madrid, 1974.
	
	\bibitem{Hir4}
	H.~Hironaka.
	\newblock {\em Idealistic exponents of singularity, \rm{Algebraic geometry (J.
			J. Sylvester Sympos., Johns Hopkins Univ., Baltimore, Md., 1976)}}.
	\newblock Johns Hopkins Univ. Press, Baltimore, Md., 1977.
	
	\bibitem{Kol}
	J.~Koll{\'a}r.
	\newblock {\em Lectures on resolution of singularities (AM-166)}, volume 166.
	\newblock Princeton University Press, 2009.
	
	\bibitem{Lip}
	J.~Lipman.
	\newblock Desingularization of two-dimensional schemes.
	\newblock {\em Ann. of Math. (2)}, 107(1):151--207, 1978.
	
	\bibitem{Mar-M}
	G.~Marzo and M.~McQuillan.
	\newblock Very functorial, very fast, and very easy resolution of
	singularities.
	\newblock {\em Geom. Funct. Anal.}, 30(3):858--909, 2020.
	
	\bibitem{McQ-P}
	M.~McQuillan and D.~Panazzolo.
	\newblock Almost \'{e}tale resolution of foliations.
	\newblock {\em J. Differential Geom.}, 95(2):279--319, 2013.
	
	\bibitem{Mol}
	B.~Molina-Samper.
	\newblock Combinatorial aspects of classical resolution of singularities.
	\newblock {\em Revista de la Real Academia de Ciencias Exactas, F{\'\i}sicas y
		Naturales. Serie A. Matem{\'a}ticas}, 113(4):3931--3948, 2019.
	
	\bibitem{Mus}
	M.~Musta\c{t}\u{a}.
	\newblock Introduction to resolution of singularities.
	\newblock In {\em Analytic and algebraic geometry}, volume~17 of {\em IAS/Park
		City Math. Ser.}, pages 405--449. Amer. Math. Soc., Providence, RI, 2010.
	
	\bibitem{Nov-S}
	J.~Novacoski and M.~Spivakovsky.
	\newblock On the local uniformization problem.
	\newblock In {\em Algebra, logic and number theory}, volume 108 of {\em Banach
		Center Publ.}, pages 231--238. Polish Acad. Sci. Inst. Math., Warsaw, 2016.
	
	\bibitem{Pan}
	D.~Panazzolo.
	\newblock Resolution of singularities of real-analytic vector fields in
	dimension three.
	\newblock {\em Acta Math.}, 197(2):167--289, 2006.
	
	\bibitem{Pil}
	O.~Piltant.
	\newblock An axiomatic version of {Z}ariski's patching theorem.
	\newblock {\em Rev. R. Acad. Cienc. Exactas F\'{\i}s. Nat. Ser. A Mat. RACSAM},
	107(1):91--121, 2013.
	
	\bibitem{Spi}
	M.~Spivakovsky.
	\newblock A solution to hironaka’s polyhedra game.
	\newblock In {\em Arithmetic and Geometry: Papers Dedicated to IR Shafarevich
		on the Occasion of His Sixtieth Birthday. Volume II: Geometry}, pages
	419--432. Springer, 1983.
	
	\bibitem{Spi1}
	M.~Spivakovsky.
	\newblock Resolution of singularities: an introduction.
	\newblock In {\em Handbook of geometry and topology of singularities. {I}},
	pages 183--242. Springer, Cham, 2020.
	
	\bibitem{Vil}
	O.~Villamayor.
	\newblock Constructiveness of {H}ironaka's resolution.
	\newblock {\em Ann. Sci. \'{E}cole Norm. Sup. (4)}, 22(1):1--32, 1989.
	
	\bibitem{Wlo2}
	J.~W{\l}odarczyk.
	\newblock Simple {H}ironaka resolution in characteristic zero.
	\newblock {\em J. Amer. Math. Soc.}, 18(4):779--822, 2005.
	
	\bibitem{Zar}
	O.~Zariski.
	\newblock Local uniformization on algebraic varieties.
	\newblock {\em Ann. of Math. (2)}, 41:852--896, 1940.
\end{thebibliography}

\end{document}